\documentclass[11pt]{article}

\usepackage[utf8]{inputenc}
\usepackage{a4wide}
\usepackage{graphicx}
\usepackage[hmarginratio={1:1},     
  vmarginratio={1:1},     
  textwidth=17cm,        
  textheight=24cm,
  heightrounded,]{geometry}
  
\usepackage{amsmath,accents}
\usepackage{amsthm}
\usepackage{amssymb}

\usepackage{mathrsfs, mathtools}

\let\emph\undefined
\newcommand{\emph}[1]{\textsl{#1}}

\usepackage{hyperref}

\numberwithin{equation}{section}
\usepackage{mathtools}
\mathtoolsset{showonlyrefs}
\numberwithin{equation}{section}

\usepackage{overpic} 
\newtheoremstyle{style1}
  {13pt}
  {13pt}
  {}
  {}
  {\normalfont\bfseries}
  {.}
  {.5em}
  {}
\theoremstyle{style1}

\newtheorem{definition}[equation]{Definition}
\newtheorem{example}[equation]{Example}
\newtheorem{remark}[equation]{Remark}

\usepackage{tocloft}

\newtheoremstyle{style2}
  {13pt}
  {13pt}
  {\slshape}
  {}
  {\normalfont\bfseries}
  {.}
  {.5em}
  {}

\theoremstyle{style2}

\newtheorem{lemma}[equation]{Lemma}
\newtheorem{theorem}[equation]{Theorem}
\newtheorem{proposition}[equation]{Proposition}

\usepackage{tikz}
\usetikzlibrary{matrix,arrows,decorations.pathmorphing}
\usepackage{tikz-cd}

\newcommand{\C}{\mathbb{C}}
\newcommand{\Z}{\mathbb{Z}}

\newcommand{\DW}{\operatorname{\normalfont DW}}

\newcommand{\Hom}{\operatorname{Hom}}

\newcommand{\id}{\text{id}}

\newcommand{\fvs}{{\mathbf{FinVect}_\C}}
\newcommand{\FinVect}{{\mathbf{FinVect}_\C}}
\newcommand{\Tvs}{{2\mathbf{Vect}_\C}}
\newcommand{\TwoVect}{{2\mathbf{Vect}_\C}}

\newcommand{\Cob}{{\mathbf{Cob}}}

\newcommand{\DS}{\text{/\hspace{-0.1cm}/}}
\newcommand{\U}{\operatorname{U}}

\let\to\undefined
\newcommand{\to}{\longrightarrow}
\let\mapsto\undefined
\newcommand{\mapsto}{\longmapsto}
\newcommand{\spr}[1]{\left( #1\right)}
\newcommand{\Fund}{\operatorname{Fund}}
\newcommand{\TCob}{T\text{-}\Cob(n,n-1,n-2)}
\newcommand{\opp}{\text{opp}}

\DeclareMathSymbol{\Phiit}{\mathalpha}{letters}{"08} 
\DeclareMathSymbol{\Psiit}{\mathalpha}{letters}{"09}
\DeclareMathSymbol{\Sigmait}{\mathalpha}{letters}{"06}
\DeclareMathSymbol{\Xiit}{\mathalpha}{letters}{"04}
\DeclareMathSymbol{\Piit}{\mathalpha}{letters}{"05}\let\Pi\undefined\newcommand{\Pi}{\Piit}
\DeclareMathSymbol{\Gammait}{\mathalpha}{letters}{"00}
\DeclareMathSymbol{\Omegait}{\mathalpha}{letters}{"0A}
\DeclareMathSymbol{\Upsilonit}{\mathalpha}{letters}{"07}
\DeclareMathSymbol{\Thetait}{\mathalpha}{letters}{"02}
\let\Omega\undefined\DeclareMathSymbol{\Omega}{\mathalpha}{letters}{"0A}
\let\Lambda\undefined\DeclareMathSymbol{\Lambda}{\mathalpha}{letters}{"03}

\let\Phi\undefined\newcommand{\Phi}{\Phiit}
\let\Sigma\undefined\newcommand{\Sigma}{\Sigmait}
\let\Psi\undefined\newcommand{\Psi}{\Psiit}
\let\Gamma\undefined\newcommand{\Gamma}{\Gammait}

\title{Extended primitive HQFTs}
\author{Notes by Lukas Müller and Lukas Woike}
\begin{document}

\begin{flushright}
\small
{\sf EMPG--18--06} \\
{\sf [ZMP-HH/18-8]} \\
	\textsf{Hamburger Beiträge zur Mathematik Nr. 727}\\
\textsf{February 2018}
\end{flushright}

\vspace{10mm}

\begin{center}
	\textbf{\LARGE{Parallel Transport of Higher Flat Gerbes as an Extended Homotopy Quantum Field Theory}}\\
	\vspace{1cm}
	{\large Lukas Müller $^{a}$} \ \ and \ \ {\large Lukas Woike $^{b}$}

\vspace{5mm}

{\em $^a$ Department of Mathematics\\
Heriot-Watt University\\
Colin Maclaurin Building, Riccarton, Edinburgh EH14 4AS, U.K.}\\
and {\em Maxwell Institute for Mathematical Sciences, Edinburgh, U.K.}\\
Email: \ {\tt lm78@hw.ac.uk \ }
\\[7pt]
{\em $^b$ Fachbereich Mathematik, \ Universit\"at Hamburg\\
Bereich Algebra und Zahlentheorie\\
Bundesstra\ss e 55, \ D\,--\,20\,146\, Hamburg }\\
Email: \ {\tt  lukas.jannik.woike@uni-hamburg.de\ }
\end{center}
\vspace{1cm}
\begin{abstract}
\noindent We prove that the parallel transport of a flat $n-1$-gerbe on any given target space gives rise to an $n$-dimensional extended homotopy quantum field theory. In case the target space is the classifying space of a finite group, we provide explicit formulae for this homotopy quantum field theory in terms of transgression. Moreover, we use the geometric theory of orbifolds to give a dimension-independent version of twisted and equivariant Dijkgraaf-Witten models. Finally, we introduce twisted equivariant Dijkgraaf-Witten theories giving us in the 3-2-1-dimensional case  a new class of equivariant modular tensor categories which can be understood as twisted versions of the equivariant modular categories constructed by Maier, Nikolaus and Schweigert.
\end{abstract}

\tableofcontents

\section{Introduction}
Homotopy quantum field theories \cite{turaevhqft}
are a flavor of topological quantum field theories defined on manifolds 
equipped with a map to some fixed topological space $T$, the target space.
More precisely, an $n$-dimensional homotopy quantum field theory $Z$ is a symmetric monoidal functor from the category of oriented compact $n$-dimensional bordisms decorated with maps landing in $T$ to the category of vector spaces. 
When spelled out, this definition entails that $Z$ assigns to a closed oriented $n$-dimensional manifold $M$ with a map $\psi :M \to T$ a complex number only depending on the homotopy class of $\psi$. 
An important example is the so-called \emph{primitive homotopy quantum field theory} constructed in \cite[I.2.1]{turaevhqft} which is based on a singular cocycle $\theta \in Z^n(T;\U(1))$ and assigns to a closed oriented $n$-dimensional manifold $M$ with a map $\psi :M \to T$
the evaluation 
\begin{align}
\langle\psi^* \theta,\mu_M \rangle  \label{wzwamplitude}
\end{align} 
of the pullback $\psi^* \theta$ on the fundamental class $\mu_M$ of $M$.
In order to define the entire homotopy quantum field theory, one has to define vector spaces for $n-1$-dimensional compact 
oriented manifolds and linear maps for $n$-dimensional compact oriented bordisms with boundary -- both equipped with maps to $T$.
These assignments are subject to compatibility with the gluing of manifolds equipped with maps on them and with disjoint union.

The compatibility with gluing takes into account manifolds of dimension $n$ and $n-1$. Of course, it is desirable to also consider manifolds of higher codimension.
This leads to the notion of an \emph{extended} topological field theory \cite{baezdolan,BDSPV15} and thereby to higher categorical manifold invariants. 
In the present text, \emph{extended} means \emph{extended once}. Such a field theory assigns quantities to manifolds of dimension $n$, $n-1$ and $n-2$ (an even further enhancement is provided by fully extended field theory; see \cite{lurietft} for an introduction and the relation to the cobordism hypothesis).
Among the extended field theories there is the subclass of \emph{invertible} theories for which the following is satisfied: 
The 2-vector spaces assigned in codimension two have one simple object 
(but \emph{without} being canonically equivalent to the category of vector spaces), 
the linear functors assigned
in codimension one are given
by tensoring with a 1-dimensional vector space and the natural transformations 
assigned in top dimension are invertible.

In this paper we answer affirmatively the question whether the homotopy quantum field theory constructed from a cocycle $\theta \in Z^n(T;\U(1))$ can be defined coherently on manifolds of codimension two in the sense that one obtains an \emph{extended} homotopy quantum field theory as defined in \cite{extofk}. 
Such a theory is defined as a symmetric monoidal functor from the bicategory $\TCob$ of compact oriented bordisms with maps to $T$ to the bicategory $\TwoVect$ of 2-vector spaces.\\[2ex]
\noindent {\textbf{Main Theorem}.} \slshape
	Any cocycle $\theta \in Z^n(T;\U(1))$ on a topological space $T$ gives rise to
	an $n$-dimensional invertible homotopy quantum field theory
	\begin{align} T_\theta : \TCob \to \TwoVect \end{align} 
	with target space $T$. \\[2ex]
\normalfont

To a compact oriented $n-2$-dimensional manifold $S$ with a map $\xi : S \to T$ the field theory $T_\theta$ assigns the 2-vector space generated by the fundamental cycles of $S$, i.e.\ the $n-2$-cocycles of $S$ representing the fundamental class of $S$. The space of morphisms from a fundamental cycle $\sigma$ to a fundamental cycle $\sigma'$ is the free vector space of the $n-1$-chains $\tau$ such that $\partial \tau =\sigma'-\sigma$ modulo the relation $\widetilde \tau \sim \langle \xi^* \theta,\lambda\rangle \tau$ for two $n-1$-chains whenever $\widetilde \tau-\tau=\partial \lambda$ for some $n$-chain $\lambda$. We then proceed and assign linear functors to $n-1$-dimensional compact oriented bordisms with boundary and maps to $T$ and natural transformations to $n$-dimensional compact oriented bordisms with corners and maps to $T$.     
Upon restriction to the endomorphisms of the empty set in $\TCob$, we recover the homotopy quantum field theory given in \cite[I.2.1]{turaevhqft}.

The construction of the extended field theory from $\theta$ is more involved than the non-extended one, but also more interesting --
 in particular from the perspective of representation theory: When specifying $T$ to be the classifying space $BG$ of a finite group $G$, we can apply the orbifold construction from  \cite{extofk} to obtain an extended $n$-dimensional topological field theory. This yields a new and very easy construction of the $\theta$-twisted Dijkgraaf-Witten theory (Section~\ref{twisteddwtheories}). For $n=3$ the value of this theory on the circle is the representation category of the $\theta$-twisted Drinfeld double of $G$, and the topological quantum field theory can be used to count the different simple representations (Proposition~\ref{satznumberofsimples}). 

Using the pushforward construction from \cite{extofk}, we can also endow the equivariant Dijkgraaf-Witten models from \cite{maiernikolausschweigerteq} with a twist, thereby obtaining a new class of examples of extended homotopy quantum field theories (Definition~\ref{defgendwtheory}) and, by evaluation of those on the circle, a new class of equivariant modular tensor categories (Theorem~\ref{Theorem: Equivariant category on S1}) which can be seen as twisted versions of the equivariant modular categories found in \cite{maiernikolausschweigerteq}.

In Section~\ref{sectransgression} we explore the relation between extended primitive homotopy quantum field theories and transgression. More precisely, we prove that for a cocycle $\theta \in Z^n(BG;\U(1))$ the associated extended field theory gives rise to a 2-line bundle over the groupoid of $G$-bundles on any closed oriented $n-2$-dimensional manifold, which is entirely determined by a transgression of $\theta$ (Theorem~\ref{Theorem: Transgression}).

The extended homotopy quantum field theory constructed in this paper has a direct physical motivation in the sense that it models the higher categorical parallel transport operator of a higher flat gerbe.
Recall that gerbes with connection are higher analogues of $\U(1)$-bundles with connection which provide the mathematical framework for the treatment of Wess-Zumino terms in two-dimensional field theory, see \cite{waldorf} for an introduction. Bundle gerbes with connection on a manifold $M$ are classified by the second hypercohomology $H^2(M;\mathscr{D}_2)$ with coefficients in the second Deligne complex. For each bundle gerbe with connection one can define a field strength, which is a  3-form on $M$; and the bundle gerbes with connection having vanishing field strength, the so-called \emph{flat gerbes}, are classified by $H^2(M;\U(1))$, i.e. by ordinary (singular) cohomology with coefficients in the constant sheaf $\U(1)$. Hence, flat gerbes can be described in purely topological terms.

More generally, we will say that for any topological space $T$ the cohomology $H^n(T;\U(1))$ classifies flat $n-1$-gerbes on $T$, i.e. a flat gerbe on $T$ can be described by a singular cocycle $\theta \in Z^n(T;\U(1))$.

The Wess-Zumino terms for the sigma model of a flat gerbe $\theta \in Z^n(T;\U(1))$ on a space $T$ can now be understood as follows: For a given map $\psi : \Sigma \to T$ from a closed oriented $n$-dimensional manifold $\Sigma$ to $T$ ($\Sigma$ can be understood as a (higher dimensional) worldsheet), the Wess-Zumino term is given  precisely by the number $\langle\psi^* \theta,\mu_\Sigma \rangle$
encountered in \eqref{wzwamplitude},
 i.e.\ by the evaluation of the pullback of $\theta$ along $\psi$ on the fundamental class $\mu_\Sigma$ of $\Sigma$. It is also referred to as the \emph{(higher) holonomy of $\theta$ with respect to $\psi$}. Of course, it is highly desirable to be able to compute the amplitude \eqref{wzwamplitude} by cutting $\Sigma$ into simpler pieces. This locality principle for the holonomy of flat gerbes was made precise in \cite[I.2.1]{turaevhqft} by the statement that a cocycle $\theta \in Z^n(T;\U(1))$ gives rise to a homotopy quantum field theory with target space $T$. The present paper establishes this locality statement one categorical level higher.

The flatness assumption throughout the article ensures that gerbes and their parallel transport can be described purely topologically. A description of non-flat gerbes is much more involved due to smoothness issues. The interpretation of the parallel transport of gerbes in the non-flat case in terms of homotopy quantum field theory has been addressed in the two-dimensional non-extended case in \cite{bunketurnerwillerton}, see also \cite{BunkWaldorf} for the open-closed case. Our results can be seen as a higher categorical, but purely topological version of this approach.

\subsection*{Acknowledgements}
We would like to thank Christoph Schweigert and Richard Szabo for the constant support of this project and helpful discussions. Furthermore, we are grateful to Severin Bunk for useful comments on a draft version of this paper.

LM is supported by the Doctoral Training Grant ST/N509099/1 from the UK Science and Technology Facilities
Council (STFC) and thanks the University of Hamburg and in particular Christoph Schweigert for their hospitality during the time when part of this work was done.
LW is supported by the RTG 1670 ``Mathematics inspired by String theory and Quantum
Field Theory''.

\section{Extended homotopy quantum field theories\label{secexthqft}}
For the notion of a homotopy quantum field theory from \cite{turaevhqft} there exists an extended version. The full definition is given in \cite{extofk} and will be recalled below.

The domain of definition for an extended homotopy quantum field theory with target space $T$ is the symmetric monoidal bordism bicategory $T\text{-}\Cob(n,n-1,n-2)$ of $T$-bordisms.

In order to define $T\text{-}\Cob(n,n-1,n-2)$, the notion of a manifold with corners is needed which we recall now based on \cite[Section~3.1.1]{schommerpries}: An \emph{$n$-dimensional manifold with corners of codimension 2} is defined to be a second countable Hausdorff space $M$ equipped with a maximal atlas of charts 
\begin{align}
M\supseteq U  \stackrel{\varphi}{\to}  V\subset \mathbb{R}^{n-2}\times (\mathbb{R}_{\ge 0})^2.
\end{align}
For a point $x\in M$ the \emph{index} of $x$ is  the number of coordinates of $(\text{pr}_{(\mathbb{R}_{\geq 0})^2}\circ \varphi)(x)$ which are equal to 0 for one chart $\varphi$ (and therefore all charts).
We define a \emph{connected face} of $M$ as the closure of a maximal connected subset of points of index 1, a \emph{face} as the disjoint union of connected faces and \emph{manifold with faces} as a manifold with corners having the property that every point of index 2 is part of precisely two different connected faces.

Lastly, we define an \emph{$n$-dimensional $\langle 2 \rangle $-manifold} as an $n$-dimensional manifold $M$ with faces equipped with a decomposition $\partial M = \partial_0 M \cup \partial_1 M$ of its topological boundary into faces having the property that $\partial_0 M \cap \partial_1 M$ is the set of corners of $M$. Here $\partial_0 M$ is called the 0-boundary of $M$ and $\partial_1 M$ the 1-boundary of $M$.

\begin{definition}\label{defmscbordcattarget}
For $n\ge 2$ and any non-empty topological space $T$, which we will refer to as the \emph{target space}, we define the bicategory $T\text{-}\Cob(n,n-1,n-2)$ as follows: 
	\begin{enumerate}
		\item[(0)] Objects are pairs $(S,\xi)$ consisting of a $n-2$-dimensional oriented closed manifold $S$ and a continuous map $\xi: S \to T$ (hereafter just referred to as map).

		\item[(1)]
		A 1-morphism $(\Sigmait,\varphi) : (S_0,\xi_0) \to (S_1,\xi_1)$ is an oriented compact collared bordism $(\Sigmait,\chi_-,\chi_+) : S_0 \to S_1$ (by this we mean a compact oriented $n-1$-dimensional manifold $\Sigmait$ with boundary equipped with orientation preserving diffeomorphisms $\chi_-\colon S_0\times [0,1) \to \Sigmait_-$ and $\chi_+:  S_1\times (-1,0] \to \Sigmait_+$ with $\Sigmait_- \cup \Sigmait_+$ being a collar of $\partial \Sigmait$),
		and a map $\varphi : \Sigmait \to T$ making the diagram
		\begin{equation}
		\begin{tikzcd}
		\, & \Sigmait \arrow{dd}{\varphi}  & \\
		S_0\times\{0\}  \arrow{ru}{\chi_-} \arrow[swap]{dr}{\xi_0\circ \operatorname{pr}_{S_0}} & & S_1\times \{0\}  \arrow[swap]{lu}{\chi_+} \arrow{dl}{\xi_1\circ \operatorname{pr}_{S_1}} \\
		& T  &
		\end{tikzcd}
		\end{equation}
		commute. No compatibility on the collars is assumed.  We define composition of 1-morphisms by gluing of bordisms along collars and maps. The collars are needed to define a smooth structure on the composition. The identities are  cylinders equipped with the trivial homotopy, where trivial means constant along the cylinder axis. 
		
\begin{figure}[hbt]\centering
\begin{overpic}[width=15.5cm]
	{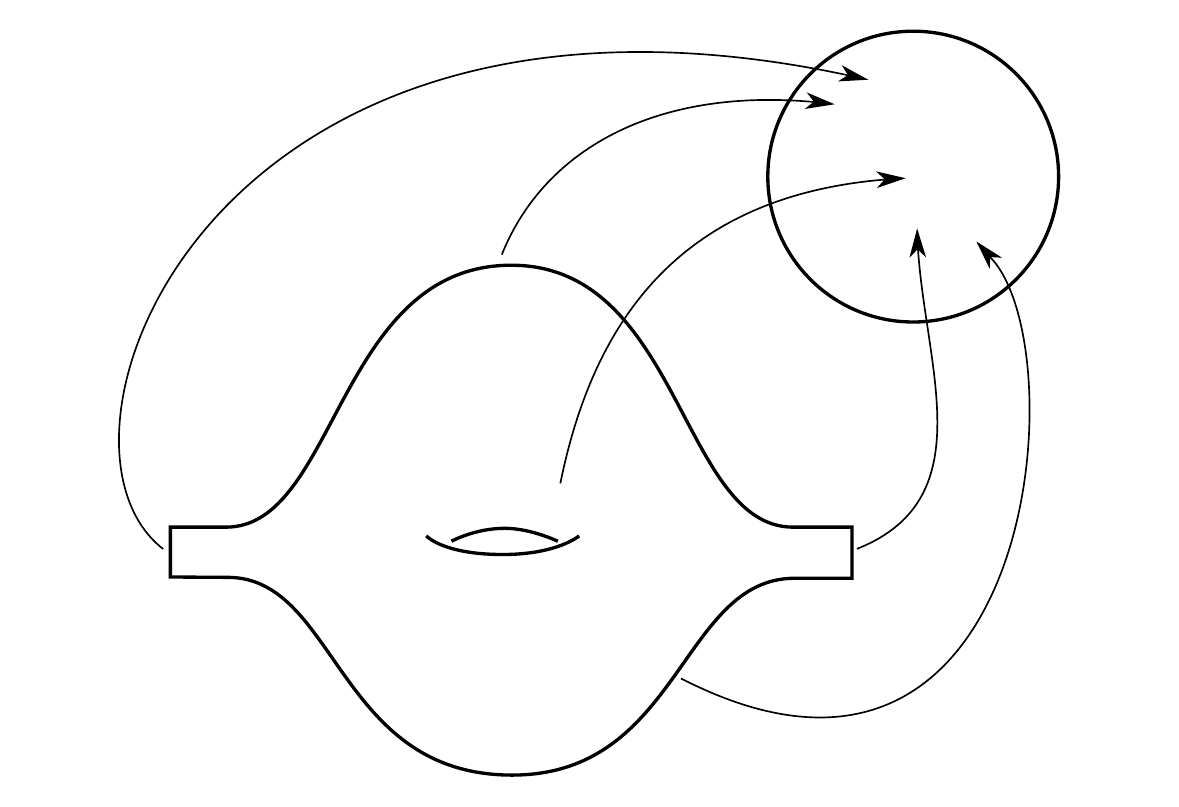}\put(79,53){$T$}\put(40,30){$M$}\put(40,0){$\Sigma_a$}
\put(34,45){$\Sigma_b$}
\put(80,10){$\varphi_a$}\put(40,51){$\varphi_b$}
\put(60,45){$\psi$}\put(0,35){$\xi_0 \times [0,1]$}
\put(68,30){$\xi_1 \times [0,1]$}
\end{overpic}
\caption{Sketch of a 2-morphism.}
\label{Fig:Sketch 2-Morphism}
\end{figure}
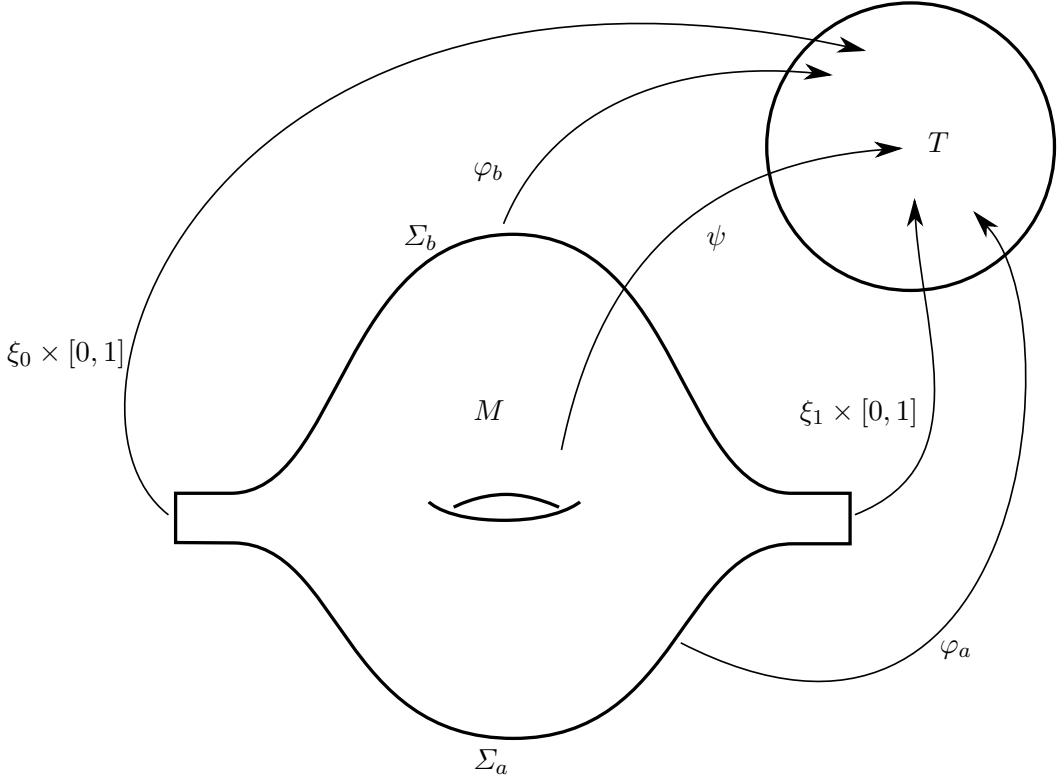

		\item[(2)]
		A 2-morphism $(\Sigmait,\varphi) \Longrightarrow (\Sigmait',\varphi')$ between 1-morphisms $(S_0,\xi_0) \to (S_1,\xi_1)$ is defined to be an equivalence class of pairs $(M,\psi)$ consisting of an $n$-dimensional collared compact oriented bordism $M : \Sigmait \to \Sigmait'$ with corners and a map $\psi : M \to T$ (see figure \ref{Fig:Sketch 2-Morphism}). Recall that an $n$-dimensional collared compact oriented bordism is a $\langle 2 \rangle$-manifold $M$ equipped with
		\begin{itemize}
			\item a decomposition of the 0-boundary $\partial_0 M = \partial_0 M_- \cup \partial_0 M_+$ and orientation preserving diffeomorphisms $\delta_- : \Sigmait \times [0,1) \to M_-$ and $\delta_+ : \Sigmait' \times (-1,0] \to M_+$ to collars of this decomposition,
			
			\item a decomposition of the 1-boundary $\partial_1 M = \partial_1 M_- \cup \partial_1 M_+$ and orientation preserving diffeomorphisms $\alpha_- : S_0 \times [0,1)\times [0,1] \to M_-$ and $\alpha_+ : S_1 \times (-1,0]\times [0,1] \to M_+$ to collars of this decomposition
		making for some $ \varepsilon >0$ the diagrams
			\begin{equation}
			\label{Condition Collars 1}
			\begin{tikzcd}
			S_0 \times [0,1)\times [0,\varepsilon) \ar{r}{\alpha_-}  \ar[swap]{rd}{\chi_-\times \id}& M & S_1 \times (-1,0]\times [0,\varepsilon) \ar[swap]{l}{\alpha_+} \ar{ld}{\chi_+\times \id} \\
			& \Sigmait \times [0,\varepsilon) \ar[swap]{u}{\delta_-}&
			\end{tikzcd},
			\end{equation}
			\begin{equation}
			\label{Condition Collars 2}
			\begin{tikzcd}
			\,          & \Sigmait' \times (-\varepsilon,0] \ar{d}{\delta_+}   & \\
			S_0 \times [0,1)\times (1-\varepsilon,1] \ar{r}{\alpha_-} \ar{ru}{\chi'_-\times \id-1}& M & S_1 \times (-1,0]\times (1-\varepsilon,1] \ar[swap]{l}{\alpha_+} \ar[swap]{lu}{\chi'_+\times \id-1}
			\end{tikzcd}
			\end{equation} and
			\[
			\begin{tikzcd}
			\,        & M \ar{dd}{\psi} & \\
			S_0\times [0,1] \sqcup \Sigmait  \ar{ru}{\alpha_- \sqcup \delta_-} \ar[swap]{dr}{\xi_0\circ \text{pr}_{S_0} \sqcup \varphi} & & S_1\times [0,1] \sqcup \Sigmait' \ar[swap]{lu}{\alpha_+ \sqcup \delta_+} \ar{dl}{\xi_1\circ \text{pr}_{S_1} \sqcup \varphi'} \\
			& T &
			\end{tikzcd}
			\]
			 commute. Again no compatibility on the collars is assumed.
		\end{itemize}

		We define two pairs $(M,\psi)$ and $(\widetilde M,\widetilde \psi)$ to be equivalent if we can find an orientation-preserving diffeomorphism $\Phiit : M \to\widetilde M$ such that the diagram
		\[
		\begin{tikzcd}
		\,        & M \ar{dd}{\Phiit} & \\
		\Sigmait \times [0,1)  \ar{ru}{\delta_-} \ar[swap]{dr}{\widetilde \delta_-} & & \Sigmait' \times (-1,0] \ar[swap]{lu}{\delta_+} \ar{dl}{\widetilde \delta_+} \\
		& \widetilde M &
		\end{tikzcd} \ ,
		\]
		and a similar diagram involving the collars of the 1-boundary commute, and if moreover $\psi = \widetilde \psi \circ \Phiit$ holds.

\end{enumerate}
In order to define the vertical composition of 2-morphisms, we fix a diffeomorphism $[0,2]\to [0,1]$ which is equal to the identity on a neighborhood of $0$ and given by $x\mapsto x-1$ in a neighborhood of $2$.
Now we define the vertical composition by gluing using the collars of 0-boundaries. We use our fixed diffeomorphism to rescale both the ingoing and outgoing 1-collars. 

We define horizontal composition of 2-morphisms by gluing manifolds and maps along 1-boundaries, where the new 0-collars arise from the old ones by restricting them to a small 
enough neighbourhood of the boundary such that \eqref{Condition Collars 1} and \eqref{Condition Collars 2} allow us to glue them along the boundary and then rescaling the interval.
By disjoint union the structure of a symmetric monoidal bicategory with duals on $T\text{-}\Cob(n,n-1,n-2)$ is obtained.
\end{definition}

\begin{definition}\label{defhqftext}
	An \emph{$n$-dimensional extended homotopy quantum field theory with target space $T$ and values in a symmetric monoidal bicategory $\mathcal{S}$} is defined to be a symmetric monoidal functor $Z: T\text{-}\Cob(n,n-1,n-2) \to \mathcal{S}$ fulfilling the property of \emph{homotopy invariance}: Given 2-morphisms $(M,\psi), (M,\psi') : (\Sigmait_a,\varphi_a) \Longrightarrow (\Sigmait_b,\varphi_b)$ between 1-morphisms $(\Sigmait_a,\varphi_a) , (\Sigmait_b,\varphi_b) : (S_0,\xi_0) \to (S_1,\xi_1)$ such that $\psi$ is homotopic to $\psi'$ relative the boundary of $M$ we require that
\begin{equation}
\begin{tikzcd}
\, & \ \ \ \ \ \  \ar[Rightarrow, swap, shorten <= 8,shorten >= 8]{dd}{Z(M,\psi)} &  \\
Z(S_0,\xi_0) \ar[bend left=50]{rr}{Z(\Sigmait_a,\varphi_a)} \ar[bend right=50, swap]{rr}{Z(\Sigmait_b,\varphi_b)}  & & Z(S_1,\xi_1) \\
\, & \, &
\end{tikzcd} 
=
\begin{tikzcd}
\, & \ \ \ \ \ \  \ar[Rightarrow, swap, shorten <= 8,shorten >= 8]{dd}{Z(M,\psi')} &  \\
Z(S_0,\xi_0) \ar[bend left=50]{rr}{Z(\Sigmait_a,\varphi_a)} \ar[bend right=50, swap]{rr}{Z(\Sigmait_b,\varphi_b)}  & & Z(S_1,\xi_1) \\
\, & \, &
\end{tikzcd} \ .
	\end{equation}
	
An extended homotopy quantum field theory $Z: T\text{-}\Cob(n,n-1,n-2) \to \mathcal{S}$ is called \emph{invertible} if it takes values in the sub-2-groupoid of $\mathcal{S}$ containing the objects, 1-morphisms and 2-morphisms of $\mathcal{S}$ which are invertible with respect to the tensor product, i.e. if it takes values in the Picard sub-2-groupoid of $\mathcal{S}$. 
\end{definition}

\begin{remark}
In this article we use the 2-category of Kapranov–Voevodsky (KV) 2-vector spaces $\Tvs$ \cite{KV} as codomain of our field theories. The codomain will also be referred to as the \emph{coefficients}. For a discussion of alternative targets see \cite[Appendix A]{BDSPV15}. The most common coefficients for extended field theories contain $\Tvs$ as a sub-2-category.
The objects of $\Tvs$ are KV 2-vector spaces, i.e. semi-simple abelian $\fvs$-module categories with a finite number of isomorphism classes of simple objects. We denote the action of a finite-dimensional complex vector space $V$ on a 2-vector space $\mathcal{X}$ by $V*? \colon \mathcal{X}\rightarrow \mathcal{X}$. The 1-morphisms are linear functors, i.e. $\fvs$-module functors, and the 2-morphisms are natural transformations. 
The Deligne product $\boxtimes$ makes $\Tvs$ into a symmetric monoidal bicategory. 
Note that every KV 2-vector space is equivalent to $(\fvs)^n$ for some integer $n \ge 0$. Using this equivalence we can express a linear functor up to natural isomorphism as a matrix with vector spaces as entries. A natural transformation can be described as a matrix with linear maps as entries \cite{Mor11}.  
\end{remark}

\section{Construction of the parallel transport homotopy quantum field theory\label{secconstruchqft}}
In this section we state our main result as Theorem~\ref{mainthm}: To a flat $n-1$-gerbe on a topological space $T$ represented by a singular cocycle $\theta \in Z^n(T;\U(1))$ we associate
an extended homotopy quantum field theory
\begin{align} T_\theta : \TCob \to \TwoVect \end{align} which can be understood as the parallel transport operator of $\theta$. The evaluation of $T_\theta$ on an $n$-dimensional closed oriented manifold $M$ together with a map to $\psi : M \to T$ yields an element in $\U(1)$, namely the holonomy of $\theta$ with respect to $\psi$.  

\subsection{Definition on objects}
Consider an object $(S,\xi)$ in $\TCob$, i.e. $S$ is a closed oriented $n-2$-dimensional manifold and $\xi : S \to T$ a continuous map.
The extended homotopy quantum field theory $T_\theta$ assigns to $(S,\xi)$ a 2-vector space $T_\theta(S,\xi)$ that we are going to define in this subsection.

To this end, denote by $\Fund(S)$ the groupoid of fundamental cycles of $S$. Its objects are fundamental cycles of $S$, i.e. those elements of $Z_{n-2}(S)$ representing the fundamental class in $H_{n-2}(S)$. A morphism $\sigma \to \sigma'$ between two fundamental cycles is an $n-1$-chain $\tau$ such that $\partial \tau = \sigma'-\sigma$. Composition is given by addition of $n-1$-chains.

The objects in $T_\theta(S,\xi)$ are generated via biproduct from the object set of $\Fund (S)$, i.e. formal finite sums $\bigoplus_{i=1}^n V_i* \sigma_i$, where the $V_i$ are finite-dimensional complex vector spaces and the $\sigma_i$ are objects in $\Fund (S)$. 
We write $\sigma$ for $\C * \sigma$. 
The space of morphisms between $\sigma,\sigma' \in \Fund(S)$ seen as objects of $T_\theta(S,\xi)$ is given by
\begin{align} 
\Hom_{T_\theta(S,\xi)}(\sigma,\sigma') := \frac{\mathbb{C}[\Hom_{\Fund(S)}(\sigma,\sigma')]}{\sim},      \label{defmorphismon0cells}
\end{align} 
where $\mathbb{C}[\Hom_{\Fund(S)}(\sigma,\sigma')]$ is the free complex vector space on the set $\Hom_{\Fund(S)}(\sigma,\sigma')$, and for two morphisms $\tau,\widetilde \tau : \sigma \to \sigma'$ we make the identification
\begin{align} \widetilde \tau \sim \langle \xi^* \theta,\lambda\rangle \tau\ ,  \label{defmorphismon0cells2} \end{align} whenever $\widetilde \tau-\tau=\partial \lambda$ for some $\lambda \in C_n(S)$. By angular brackets we denote the evaluation of cochains on chains. 
Note that in \eqref{defmorphismon0cells2} the choice of $\lambda$ does not matter. 
In order to obtain the morphism spaces between all objects in $T_\theta(S,\xi)$, \eqref{defmorphismon0cells} has to be extended bilinearly, i.e.
\begin{align}
\Hom_{T_\theta(S,\xi)} \left( \bigoplus_{i=1}^n V_i * \sigma_i\ ,\ \bigoplus_{j=1}^m V_j * \sigma_j \right) = \bigoplus_{i,j} \Hom(V_i,V_j) \otimes \Hom_{T_\theta(S,\xi)} (\sigma_i, \sigma_j)
\end{align} for all formal finite sums. 
Composition is defined by matrix multiplication and composition in $\Fund (S)$. 

The $\fvs$-module structure is given by
\begin{align}
* \colon \fvs\times T_{\theta}(S,\xi) &\longrightarrow T_{\theta}(S,\xi) \\
 V\times \left(\bigoplus_{i=1}^n V_i * \sigma_i\right) & \mapsto \left(\bigoplus_{i=1}^n (V\otimes V_i) * \sigma_i\right) \ .  
\end{align}
It is now easy to see that $T_\theta(S,\xi)$ is a 2-vector space with one simple object, i.e. a 2-line.

\subsection{Definition on 1-morphisms}
Let $(\Sigmait,\varphi) : (S_0,\xi_0) \to (S_1,\xi_1)$ be a 1-morphism in $\TCob$. Again, we denote by $\Fund(\Sigmait)$ the groupoid of fundamental cycles of $\Sigmait$, i.e. the groupoid of relative cycles in $C_{n-1}(\Sigmait)$ representing the fundamental class of $\Sigmait$ in $H_{n-1}(\Sigmait,\partial \Sigmait)$. For fundamental cycles $\sigma_0$ and $\sigma_1$ of $S_0$ and $S_1$, respectively, we denote by $\Fund_{\sigma_0}^{\sigma_1}(\Sigmait)$ the subgroupoid of $\Fund(\Sigmait)$ spanned by all fundamental cycles $\mu$ of $\Sigmait$ with $\partial \mu = \sigma_1-\sigma_0$. Here we suppress the inclusion of the ingoing and outgoing boundary into $\Sigmait$ in the notation. By \cite[VI.,~Lemma~9.1]{bredon} the groupoid $\Fund_{\sigma_0}^{\sigma_1}(\Sigmait)$ is non-empty and connected. 

In order to define the 2-linear map $T_\theta(\Sigmait,\varphi) : T_\theta (S_0,\xi_0) \to T_\theta(S_1,\xi_1)$ we define on the free vector space $\mathbb{C}[\Fund_{\sigma_0}^{\sigma_1}(\Sigmait)]$ the equivalence relation
\begin{align}
\mu' \sim \langle\varphi^* \theta ,\nu\rangle \mu 
\end{align}
 for any $\nu \in C_n(\Sigmait)$ such that $\partial \nu =\mu-\mu'$. We use the notation
\begin{align}
\Sigma^\varphi \spr{\sigma_1, \sigma_0} := \frac{\mathbb{C}[\Fund_{\sigma_0}^{\sigma_1}(\Sigmait)]}{\sim}
\end{align} for the quotient and observe that $\Sigma^\varphi \spr{?,\sigma_0}$ extends to a functor $\Fund^{\opp}(S_1)\to \fvs$, which is
defined on a morphism $\lambda : \sigma_1 \to \sigma_2$ in $\Fund(S_1)$ by 
\begin{align}
\Sigma^\varphi \spr{?,\sigma_0} (\lambda):  \Sigma^\varphi \spr{\sigma_2,\sigma_0} & \to \Sigma^\varphi \spr{\sigma_1,\sigma_0}   \\
\mu &\mapsto \mu- \lambda \ .
\end{align}
A straightforward calculation shows that this is well-defined.

Now $T_\theta(\Sigmait,\varphi) : T_\theta (S_0,\xi_0) \to T_\theta(S_1,\xi_1)$ is defined on objects by the coend
\begin{align} 
T_\theta(\Sigmait,\varphi) \sigma_0 := \int^{\sigma_1 \in \Fund(S_1)} \Sigmait^\varphi \spr{\sigma_1,\sigma_0} * \sigma_1 
\end{align} 
and linear extension. For the necessary background on coends we refer to \cite[Chapter IX.6]{MacLane}.
Here, the coend can be replaced by an end, since it is taken over a groupoid and limits and colimits over essentially finite groupoids taken in a 2-vector space coincide.

Before completing the definition of $T_\theta$ we already show that it respects the composition of 1-morphisms up to coherent natural isomorphism. This is a crucial part of the 2-functoriality of $T_\theta$. For the proof we will need the following Gluing Lemma. It is a special case of the Gluing Lemma for $\langle 2\rangle$-manifolds, which will appear as Lemma~\ref{lemmagluing2man} below, hence we omit the proof. In the statement we suppress the inclusions of boundary components in the notation.

\begin{lemma}[Gluing Lemma for manifolds with boundary]\label{lemmamscgluingcyclescob}
	Let $\Sigmait_a \colon S_0 \to S_1$ and $\Sigmait_b \colon S_1 \to S_2$ be 1-morphisms in $\Cob(n,n-1,n-2)$, and let $\nu \in C_{n-1}(\Sigmait_a)$ and $\nu' \in C_{n-1}(\Sigmait_b)$ be fundamental cycles with $\partial \nu =  \sigma_1 -  \sigma_0$ and $\partial \nu' =  \sigma_2 -  \sigma_1$ for fixed fundamental cycles $\sigma_j \in Z_{n-2}(S_j)$, $j=0,1,2$. Then by $\nu'\circ \nu:= \nu'+\nu \in  C_{n-1}(\Sigmait_b \circ \Sigmait_a) $ we get a fundamental cycle of $\Sigmait_b \circ \Sigmait_a$ satisfying $\partial (\nu'\circ \nu) =  \sigma_2 -  \sigma_0$.
\end{lemma}

\begin{lemma}\label{lemmacompof1mor}
	$T_\theta$ respects the composition of 1-morphisms up to coherent natural isomorphism.
	\end{lemma}

\begin{proof}
The coherence isomorphisms consist of natural isomorphisms 
\begin{align}\label{eqnnatisocohuni}
\Phiit_{(S,\xi)}\colon \id_{T_\theta(S,\xi)} \Longrightarrow T_\theta(\id_{S,\xi}) 
\end{align}
for all objects $(S,\xi)\in \TCob$ and \begin{align}
\Phiit_{(\Sigmait_a,\varphi_a),(\Sigmait_b,\varphi_b)}\colon T_\theta(\Sigmait_b,\varphi_b) \circ T_\theta(\Sigmait_a,\varphi_a) \Longrightarrow T_\theta((\Sigmait_b,\varphi_b)\circ (\Sigmait_a,\varphi_a))\label{eqnnatisocohasso}
\end{align}
for all composable 1-morphisms $(\Sigmait_a,\varphi_a) : (S_0,\xi_0) \to (S_1,\xi_1)$ and $(\Sigmait_b,\varphi_b) : (S_1,\xi_1) \to (S_2,\xi_2)$ in $\TCob$.

Using the enriched co-Yoneda lemma we can write the identity as the coend
\begin{align}
\id_{T_\theta(S,\xi)}(?) \cong \int^{\sigma \in {T_\theta(S,\xi)}} \Hom_{T_\theta(S,\xi)} (\sigma , ?)* \sigma .
\end{align} 
Without loss of generality, we can evaluate this at a generator $\sigma_0\in \Fund(S_1)$
\begin{align}
\sigma_0 \cong \int^{\sigma \in {T_\theta(S,\xi)}} \Hom _{T_\theta(S,\xi)} ( \sigma , \sigma_0)* \sigma \cong \int^{\sigma \in {\Fund(S)}} \Hom _{T_\theta(S,\xi)}( \sigma , \sigma_0)* \sigma.\label{cohid1eqn}
\end{align}  
On the other hand, we have
\begin{align}
T_\theta(\id_{S,\xi})(\sigma_0) = \int ^{\sigma \in {\Fund(S)}}  (S\times [0,1])^{ \xi \times [0,1]}  (\sigma , \sigma_0)* .  \sigma \label{cohid2eqn}
\end{align}
There is a natural isomorphism
\begin{align}
(S\times [0,1])^{ \xi \times [0,1]} (\sigma , \sigma_0) & \to \Hom ( \sigma , \sigma_0) \\
 \mu & \mapsto - {p_S}_* \mu 
\end{align}
using the projection $p_S:S \times [0,1]\to S$. It induces an isomorphism  between the coends in \eqref{cohid1eqn} and \eqref{cohid2eqn} and gives us the desired isomorphism \eqref{eqnnatisocohuni}.

To specify the natural isomorphism \eqref{eqnnatisocohasso} for 1-morphisms $(\Sigmait_a,\varphi_a) : (S_0,\xi_0) \to (S_1,\xi_1)$ and $(\Sigmait_b,\varphi_b) : (S_1,\xi_1) \to (S_2,\xi_2)$ in $\TCob$, we note that for a generator $\sigma_0 \in \Fund(S_0)$
\begin{align}
(T_\theta(\Sigmait_b,\varphi_b) \circ T_\theta(\Sigmait_a,\varphi_a)) \sigma_0 &= \int^{(\sigma_1,\sigma_2) \in \Fund (S_1)\times \Fund (S_2)}  \Sigmait^{\varphi_a}_a\spr{ {\sigma_1},{\sigma_0} } \otimes \Sigmait_b^{\varphi_b}\spr{ {\sigma_2},{\sigma_1}}  * \sigma_2 \  \\ &\cong \int^{\sigma_2 \in \Fund (S_2)} \left(\int^{\sigma_1 \in \Fund (S_1)}  \Sigmait^{\varphi_a}_a \spr{ {\sigma_1},{\sigma_0} } \otimes \Sigmait_b^{\varphi_b}\spr{ {\sigma_2},{\sigma_1}} \right) * \sigma_2,\label{eqnrefinnercoend}
\end{align} where we have used Fubini's Theorem for coends. 
The action of a morphism $\lambda_1 :  \sigma_1 \to \sigma'_1 \in \Fund (S_1)$ in the covariant component is induced by
\begin{align}
\Sigmait_b^{\varphi_b}\spr{\sigma_2, \sigma_1} & \to \Sigmait_b^{\varphi_b} \spr{\sigma_2,\sigma'_1} \\
\mu & \mapsto \mu - \lambda_1.
\end{align}
To compute the inner coend in \eqref{eqnrefinnercoend}, we observe that 
\begin{align}
\Phiit_{\sigma_1,\sigma_2} :  \Sigmait_a^{\varphi_a} \spr{\sigma_0,\sigma_1}\otimes \Sigmait_b^{\varphi_b}\spr{\sigma_1,\sigma_2} & \to (\Sigmait_b \circ \Sigmait_a)^{\varphi_b \cup \varphi_a} \spr{ \sigma_2,\sigma_0}  \\
\mu_1\otimes \mu_2 & \mapsto \mu_1 +\mu_2 \ .
\end{align} is a canonical isomorphism by Lemma~\ref{lemmamscgluingcyclescob}. Here $\varphi_b \cup \varphi_a : \Sigmait_b \circ \Sigmait_a\to T$ is the map obtained from gluing $\varphi_a$ and $\varphi_b$. 
We now obtain
\begin{align} \int^{\sigma_1 \in \Fund (S_1)} \Sigmait_a^{\varphi_a} \spr{ {\sigma_1},{\sigma_0} } \otimes \Sigmait_b^{\varphi_b}\spr{{\sigma_2},{\sigma_1}} &\cong \int^{\sigma_1 \in \Fund (S_1)} (\Sigmait_b \circ \Sigmait_a)^{\varphi_b \cup \varphi_a} \spr{ \sigma_2,\sigma_0}  \\&\cong  \int^{\sigma_1 \in \Fund (S_1)} (\Sigmait_b \circ \Sigmait_a)^{\varphi_b \cup \varphi_a} \spr{ \sigma_2,\sigma_0} \otimes \mathbb{C}  \\&\cong (\Sigmait_b \circ \Sigmait_a)^{\varphi_b \cup \varphi_a} \spr{ \sigma_2,\sigma_0}, 
\end{align} 
where in the last step we used that $\Fund (S_1)$ is connected. Insertion into \eqref{eqnrefinnercoend} yields isomorphisms
\begin{align}
(T_\theta(\Sigmait_b,\varphi_b) \circ T_\theta(\Sigmait_a,\varphi_a)) \sigma_0 \cong (T_\theta( (\Sigmait_b,\varphi_b) \circ (\Sigmait_a,\varphi_a)) \sigma_0,
\end{align} which give us after linear extension the natural isomorphism \eqref{eqnnatisocohasso}.
\end{proof}

\subsection{Definition on 2-morphisms}
To a 2-morphism $(M,\psi): (\Sigmait_a,\varphi_a) \Longrightarrow (\Sigmait_b,\varphi_b)$ between 1-morphisms $(\Sigmait_a,\varphi_a) , (\Sigmait_b,\varphi_b) : (S_0,\xi_0) \to (S_1,\xi_1)$ in $\TCob$ we assign the 2-morphism $T_\theta(\Sigmait_a,\varphi_a) \Longrightarrow T_\theta(\Sigmait_b,\varphi_b)$ between the 1-morphisms $T_\theta(\Sigmait_a,\varphi_a), T_\theta(\Sigmait_b,\varphi_b) : T_\theta(S_0,\xi_0) \to T_\theta(S_1,\xi_1)$ consisting of the natural maps
\begin{align}
T_\theta(\Sigmait_a,\varphi_a) \sigma_0 \to T_\theta(\Sigmait_b,\varphi_b) \sigma_0
\end{align} for $\sigma_0 \in \Fund(S_0)$ which are the maps between the respective coends induced by the linear maps
\begin{align}
T_\theta(M)_{\sigma_1,\sigma_0} : \Sigma^{\varphi_a}_a(\sigma_1,\sigma_0)  \to\Sigma^{\varphi_b}_b(\sigma_1,\sigma_0) \label{eqnmapbetweencoendsinduce}  
\end{align} defined as follows: For $\mu_a \in \Fund_{\sigma_0}^{\sigma_1}(\Sigmait_a)$  we can find a fundamental cycle 
$\nu$ of $M$ with 

\begin{align}\label{EQ: Condition Fundamental cycle on manifolds with corners}
 \partial \nu = \mu_b - \mu_a +(-1)^{n-2}(\sigma_0\times [0,1]-\sigma_1\times [0,1])
\end{align} 
for some fundamental cycle $\mu_b \in \Fund_{\sigma_0}^{\sigma_1}(\Sigmait_b)$.
Mapping $\mu_a$ to $\langle \psi^* \theta,\nu\rangle [\mu_b]$ yields a well-defined linear map
$\mathbb{C}[\Fund_{\sigma_0}^{\sigma_1}(\Sigmait_a)] \to \Sigma_b^{\varphi_b}(\sigma_1,\sigma_0)$, which descends to $\Sigma_a^{\varphi_a}(\sigma_1,\sigma_0)$ and gives us the needed map \eqref{eqnmapbetweencoendsinduce}. 

In order to prove that $T_\theta$ strictly preserves the vertical composition of 2-morphisms, we need a Gluing Lemma for $\langle 2\rangle$-manifolds:

\begin{lemma}[Gluing Lemma for $\langle 2\rangle$-manifolds]\label{lemmagluing2man}
Consider two $n$-dimensional $\langle 2 \rangle$-manifolds $M_1$ and $M_2$ with representatives for the orientation class $\nu_1$ and $\nu_2$ such that $\partial \nu_i = \mu_{i,0} + \mu_{i,1}$ for $i=1,2$, where $\mu_{i,0}$ and $\mu_{i,1}$ are representatives for the fundamental class of the 0 and 1 boundary, respectively (note that this implies $\partial \mu_{i,0}= -\partial \mu_{i,1}$). 
Now assume we have an orientation reversing diffeomorphism from a connected component $\Sigma$ of, say, the 1 boundary of $M_1$ onto the 1 boundary of $M_2$ compatible with the fundamental cycles picked above, i.e. $\mu_{0,1}|_{\Sigma}= - \mu_{1,1}|_{\Sigma}$. Then $\nu_1+\nu_2$ is a representative of the fundamental class of the manifold obtained by gluing $M_1$ and $M_2$ (see figure \ref{Fig:Sketch Composition}) along $\Sigma$. 
\end{lemma}

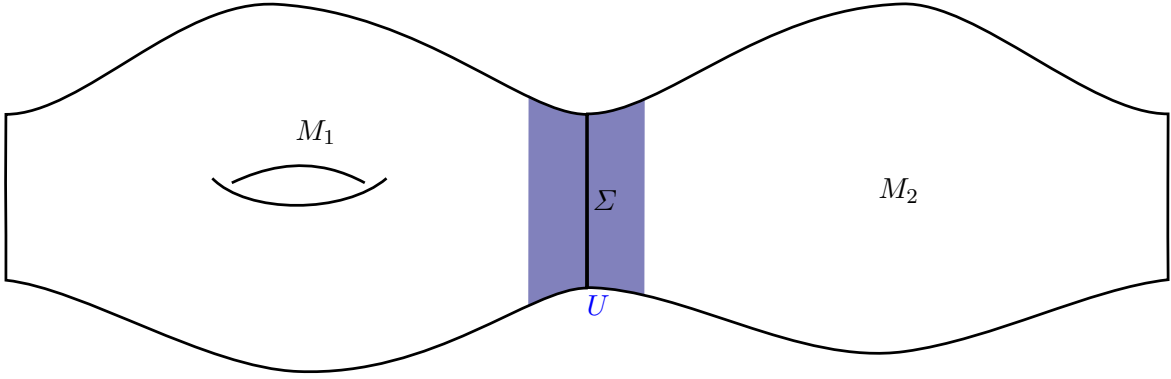
\begin{figure}[htb]\centering
\begin{overpic}[width=15.5cm
]{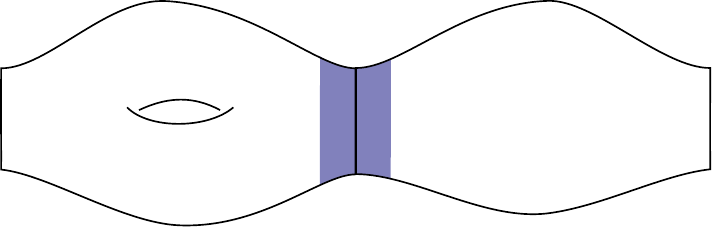}\put(25,20){$M_1$}\put(75,15){$M_2$}\put(50,5){${\color{blue}U}$}\put(50.5,14){$\Sigma$}\end{overpic}\caption{Sketch of the manifolds involved in lemma 2.8.}\label{Fig:Sketch Composition}\end{figure}

\begin{proof}
We denote the  composition of $M_1$ and $M_2$ by $M$.
Obviously, $\nu_1+\nu_2$ is a cycle relative $\partial M$. We have to show that it represents the orientation class of $M$.
To this end, we use the long exact sequence 
\begin{align}\label{eqnles1}
\dots \rightarrow H_n(\partial M \cup \Sigma, \partial M) \rightarrow H_{n}(M,\partial M)\rightarrow H_{n}(M, \partial M \cup \Sigma)\rightarrow H_{n-1}(\partial M \cup \Sigma, \partial M)\rightarrow \dots
\end{align}
in homology associated to the triple $\partial M \subset \partial M \cup \Sigma \subset M$ and compute the relevant terms occurring in it by means of a collar $U \cong \Sigma \times (-1,1)$ in $M$ (the collar exists by \cite[Lemma 2.1.6]{Lau00}; see figure \ref{Fig:Sketch Composition} for a pictorial presentation):
\begin{itemize}
	\item For the computation of $H_*(\partial M \cup \Sigma, \partial M)$ we define $V:= \partial M \cap U$ and find by excision of the complement $W:= \partial M \setminus V$ of $V$ in $\partial M$
	\begin{align}
	H_*(\partial M \cup \Sigma, \partial M) \cong H_*((\partial M \cup \Sigma) \setminus W, \partial M \setminus W). 
	\end{align} Since the inclusion $(\Sigmait,\partial \Sigma) \to ((\partial M \cup \Sigma) \setminus W, \partial M \setminus W)$ is a homotopy equivalence, we arrive at
	\begin{align} H_*(\partial M \cup \Sigma, \partial M) \cong H_*(\Sigma,\partial \Sigma).\label{eqngluing1}\end{align} 
	
	\item For the computation of $H_*(M,\partial M \cup \Sigma)$ we use that the inclusion $(M,\partial M \cup \Sigma) \to (M,\partial M \cup U)$ is a homotopy equivalence. After excising $\Sigma$ in $(M,\partial M \cup U)$ we find
	\begin{align}
	H_*(M,\partial M \cup \Sigma) \cong H_*(M_0 \setminus \Sigma,\partial M_0 \cup U_-) \oplus H_*(M_1 \setminus \Sigma,\partial M_1 \cup U_+),
	\end{align} where $U_-$ and $U_+$ is the image of $\Sigma \times (-1,0)$ and $\Sigma \times (0,1)$ in $U$, respectively. Since the inclusion $(M_0 \setminus \Sigma,\partial M_0 \cup U_-) \to (M_0,\partial M_0 \cup U_- \cup \Sigma)$ induces an isomorphism in homology and since $\Sigmait \to U_-\cup \Sigmait$ is a homotopy equivalence, we obtain $ H_*(M_0 \setminus \Sigma,\partial M_0 \cup U_-)\cong H_*(M_0,\partial M_0)$ and analogously $ H_*(M_1 \setminus \Sigma,\partial M_1 \cup U_+) \cong  H_*(M_1 ,\partial M_1 )$. Thus, we are left with
	\begin{align} H_*(M,\partial M \cup \Sigma) \cong H_*(M_0,\partial M_0) \oplus H_*(M_1 ,\partial M_1 ).\label{eqngluing2}\end{align}

 \end{itemize}
Using \eqref{eqngluing1} and \eqref{eqngluing2} we obtain from \eqref{eqnles1} the exact sequence
\begin{align}
 0 \to H_{n}(M,\partial M)\to H_{n}(M_0, \partial M_0) \oplus H_{n}(M_1, \partial M_1)\to H_{n-1}(\Sigma, \partial \Sigma) \ \    , 
\end{align} where the morphism $H_{n}(M_0, \partial M_0) \oplus H_{n}(M_1, \partial M_1)\to H_{n-1}(\Sigma, \partial \Sigma)$ takes $[\nu_1]\oplus [\nu_2]$ to $[\partial \nu_1 +\partial \nu_2]$. 
Evaluating the kernel of this morphism yields an isomorphism $H_{n}(M,\partial M)\cong \Z ([\nu_1]\oplus [\nu_2])$. This shows that $\nu_0+\nu_1$ is a generator of $H_{n}(M,\partial M)$, i.e. an orientation of $M$. This orientation agrees with the orientation of $M$ in a neighborhood of an arbitrary point away from the gluing boundary, hence they agree.
\end{proof}

\begin{lemma}\label{lemmaverticalcomp}
	$T_\theta$ preserves the vertical composition of 2-morphisms strictly.
\end{lemma}

\begin{proof}
	Given two 2-morphisms $(M,\psi) :  (\Sigmait_a,\varphi_a) \Longrightarrow (\Sigmait_b,\varphi_b)$ and $(M',\psi') \colon (\Sigmait_b,\varphi_b) \Longrightarrow (\Sigmait_c,\varphi_c)$ between 1-morphisms $(S_0,\xi_0) \to (S_1,\xi_1)$ it suffices to show that for fundamental cycles $\sigma_0$ and $\sigma_1$ of $S_0$ and $S_1$, respectively, the composition of linear maps
\begin{align}
	\Sigmait^{\varphi_a}_a (\sigma_1,\sigma_0) \xrightarrow{T_\theta (M)_{\sigma_1,\sigma_0}} \Sigmait^{\varphi_b}_b (\sigma_1,\sigma_0)   \xrightarrow{T_\theta (M')_{\sigma_1,\sigma_0}} \Sigmait^{\varphi_c}_c (\sigma_1,\sigma_0) \end{align} as defined in \eqref{eqnmapbetweencoendsinduce} is equal to
		\begin{align}
		\Sigmait^{\varphi_a}_a (\sigma_1,\sigma_0) \xrightarrow{T_\theta (M'\circ M)_{\sigma_1,\sigma_0}}  \Sigmait^{\varphi_c}_c (\sigma_1,\sigma_0) .
		\end{align} 
	Picking fundamental cycles $\nu$ and $\nu'$ for $M$ and $M'$ as in Lemma~\ref{lemmagluing2man} this follows from \begin{align}
	\langle \psi^* \theta ,\nu\rangle \cdot \langle {\psi'}^* \theta,\nu'\rangle  =\langle   (\psi'\cup \psi)^* \theta,\nu+\nu' \rangle,
	\end{align} where $\psi' \cup \psi : M'\circ M \to T$ is the map obtained by gluing $\psi$ and $\psi'$.
\end{proof}

\begin{lemma}\label{lemmahomotopyinvariance}
	 $T_\theta$ is defined such that it fulfills the axiom of homotopy invariance from Definition~\ref{defhqftext}. 
\end{lemma}

\begin{proof}
Consider 2-morphisms $(M,\psi), (M,\psi') \colon  (\Sigmait_a,\varphi_a) \Longrightarrow (\Sigmait_b,\varphi_b)$ between 1-morphisms $(S_0,\xi_0) \to (S_1,\xi_1)$ with $\psi \stackrel{h}{\simeq} \psi'$ relative $\partial M$. Let $\sigma_0$ and $\sigma_1$ be fundamental cycles of $S_0$ and $S_1$, respectively. Now for $\mu_a \in \Fund_{\sigma_0}^{\sigma_1} (\Sigma_a)$ and $\mu_b \in\Fund_{\sigma_0}^{\sigma_1} (\Sigma_b)$ we find a fundamental cycle $\nu$ of $M$ adapted to $\mu_a$ and $\mu_b$ as in equation \eqref{EQ: Condition Fundamental cycle on manifolds with corners}. By definition of $T_\theta$ on 2-morphisms it suffices to show

\begin{align}
\langle \psi^*\theta,\nu\rangle = \langle {\psi'}^*\theta,\nu\rangle . 
\end{align} 
Indeed, if we see $h$ as a map defined on $M\times [0,1]$, we find a chain homotopy between the chain maps $\psi_*$ and $\psi'_*$ given by

\begin{align} 
{\psi_*}_{p}- {\psi'_*}_{p} = \partial H_p + H_{p-1} \partial \quad \text{for all $p\in\mathbb{Z}$}, 
\end{align}
where $H_p := {h_*}_{p+1} D_p$ and 

\begin{align} 
D_p : S_p(M) \to S_{p+1}(M\times [0,1]), \quad c \mapsto c\times [0,1] 
\end{align}
 is defined using the cross-product on singular chains, see \cite[IV.16]{bredon}. Hence,

\begin{align}
\langle \psi^*\theta,\nu\rangle - \langle {\psi'}^*\theta,\nu\rangle = \langle  \theta, H_{n-1} \partial \nu  \rangle = \langle h^* \theta, \partial\nu \times [0,1] \rangle .\label{eqnhtpinvvanish}
\end{align}
The homotopy $h$ being stationary on the boundary entails
\begin{align} 
h|_{\partial M \times [0,1]}= \psi \circ p_{\partial M}
\end{align} 
with the projection $p_{\partial M} \colon \partial M \times [0,1] \to \partial M$. This yields
	\begin{align}
\langle h^* \theta, \partial\nu \times [0,1] \rangle = \langle \psi^* \theta, {p_{\partial M}}_* (\partial\nu \times [0,1]) \rangle . 	\end{align} 
We have 

\begin{align}
 \partial \left( {p_{\partial M}}_* (\partial\nu \times [0,1])\right) = 0,
\end{align}
where we use that the boundaries corresponding to the $[0,1]$ part cancel under the projection, i.e. 
\begin{align}
{p_X}_* ( X \times \partial [0,1])= 0,
\end{align} 	
for any space $X$ with projection $p_X : X \times [0,1] \to X$.
This shows that ${p_{\partial M}}_* (\partial\nu \times [0,1])$ is a cycle.
For dimensional reasons it must be a boundary as well. This shows that \eqref{eqnhtpinvvanish} vanishes and finishes the proof.
\end{proof}

\subsection{The Main Theorem}
We can now state our main result:

\begin{theorem}\label{mainthm}
	For any topological space $T$ and any cocycle $\theta \in Z^n(T;\U(1))$,
	\begin{align} T_\theta : \TCob \to \TwoVect \end{align} is an invertible homotopy quantum field theory with target $T$.
	\end{theorem}

\begin{proof}
	Thanks to the Lemmata~\ref{lemmacompof1mor}, \ref{lemmaverticalcomp} and \ref{lemmahomotopyinvariance} it remains to prove the following: 
	\begin{itemize}

		\item 
		Horizontal composition:
		Given 1-morphisms
		\begin{align}
		(\Sigmait_a,\varphi_a) : (S_0,\xi_0) &\to (S_1,\xi_1),\\
		(\Sigmait_b,\varphi_b) : (S_1,\xi_1) &\to (S_2,\xi_2)
		\end{align} and 2-morphisms 
		\begin{align}
		(M,\psi) : 	(\Sigmait_a,\varphi_a) &\Longrightarrow 	(\Sigmait_b,\varphi_b),\\
				(M',\psi') : 	(\Sigmait_b,\varphi_c) &\Longrightarrow 	(\Sigmait_c,\varphi_c)
		\end{align} we have to show that for fundamental cycles $\sigma_0$, $\sigma_1$ and $\sigma_2$ of $S_0$, $S_1$ and $S_2$, respectively, the square 
		\begin{equation}
		\begin{tikzcd}
		\Sigmait_a^{\varphi_a} (\sigma_1,\sigma_0)\otimes \Sigmait_b^{\varphi_b} (\sigma_2,\sigma_1) \ar{rr}{\Phi} \ar[swap]{dd}{T_\theta (M)_{\sigma_1,\sigma_0}\otimes T_\theta (M')_{\sigma_2,\sigma_1}} & & (\Sigmait_b \circ \Sigmait_a)^{\varphi_b\cup \varphi_a} (\sigma_2,\sigma_1) \ar{dd}{T_\theta (M' \circ M)_{\sigma_2,\sigma_0}} \\
		& & \\
		\Sigmait_b^{\varphi_b} (\sigma_1,\sigma_0)\otimes \Sigmait_c^{\varphi_c} (\sigma_2,\sigma_1) \ar{rr}{\Phi} & & (\Sigmait_c \circ \Sigmait_b)^{\varphi_c\cup \varphi_b} (\sigma_2,\sigma_1)
		\end{tikzcd}
		\end{equation} featuring as the horizontal arrows the isomorphisms from the proof of Lemma~\ref{lemmacompof1mor} commutes.
		By picking representatives for the fundamental classes as in Lemma~\ref{lemmagluing2man} this can be verified directly.
		
		\item Symmetric monoidal structure:
		There are natural equivalences of categories 
		\begin{align}
		\iota_\theta : T_\theta (\emptyset) &\longrightarrow \fvs \\
		\sigma_\emptyset &\mapsto \C
		\end{align}	
		and 	  
		\begin{align}
		\chi_\theta((S_0, \xi_0),(S_1,\xi_1))\colon T_\theta (S_0,\xi_0)\boxtimes  T_\theta (S_1,\xi_1)  &\longrightarrow T_\theta(S_0\sqcup S_1, \xi_0 \sqcup \xi_1) \\
		\sigma_{S_0}\boxtimes \sigma_{S_1} & \mapsto \sigma_{S_0} + \sigma_{S_1} ,
		\end{align}
		where we suppress the inclusion into the disjoint union and denote by $\sigma_\emptyset$ the unique fundamental cycle of the empty set (that it has by convention).
		The modifications which are part of the structure of a symmetric monoidal 2-functor (Definition B.12 \cite{MS}) are trivial, since the corresponding diagrams commute on generators.
		The simple form of the coherence isomorphism makes it straightforward to check that the corresponding diagrams commute.
	\end{itemize}
The field theory $T_\theta$ is obviously invertible.
\end{proof}

\begin{remark}\label{Rem: Non extended theory}
Restricting $T_\theta $ to the endomorphisms of the empty set induces a non-extended homotopy quantum field theory 
\begin{align}
T_\theta :  T\text{-}\mathbf{Cob}(n,n-1) \to \fvs \ ,
\end{align}
which admits the following concrete description
\begin{itemize}
\item 
To a closed $n-1$ dimensional manifold $\Sigma$ equipped with a map $\varphi : \Sigma \to T$ it assigns the vector space $T_\theta(\Sigma, \varphi)=\Sigma^\varphi(\emptyset,\emptyset)= \C[\Fund(\Sigma)]/{\sim}$.

\item 
To a morphism $(M,\psi): (\Sigma_a,\varphi_a) \to (\Sigma_b,\varphi_b)$ it assigns the linear map 
\begin{align}
T_\theta(M,\psi): T_\theta(\Sigma_a,\varphi_a) & \to T_\theta (\Sigma_b,\varphi_b) \\
[\sigma_{\Sigma_a}]& \longmapsto \langle \psi^*\theta, \sigma_M \rangle [\sigma_{\Sigma_b}] \ ,
\end{align}
with $\partial \sigma_M= \sigma_{\Sigma_b}- \sigma_{\Sigma_a}$.
\end{itemize}
This is the primitive homotopy quantum field theory constructed in \cite[I.2.1]{turaevhqft}.
\end{remark}

The following assertion shows that up to natural equivalence $T_\theta$ only depends on the cohomo\-logy
class of $\theta$. 
\begin{proposition}
	Let $\theta$ and $\theta'$ be $n$-cocycles on a topological space $T$ with values in $\U(1)$ and 
	$\Lambda$ an $n-1$-chain on $T$ satisfying $\operatorname{d}\Lambda= \theta'-\theta$. Then $\Lambda$ induces a 
	symmetric monoidal natural equivalence 
	\begin{align}
	T_\Lambda : T_\theta \longrightarrow T_{\theta'} \ \ .
	\end{align} 
\end{proposition}

\begin{proof}
	
	For all $(S,\xi)\in \TCob$ we define linear functors 
	\begin{align}
	T_\Lambda(S,\xi) : T_\theta(S,\xi) & \longrightarrow T_{\theta'}(S,\xi) \\
	\sigma & \longmapsto \sigma \\
	[\lambda] &\mapsto [\langle \xi^*\Lambda, \lambda \rangle \cdot \lambda] \ \ . 
	\end{align}
	For a 1-morphism $(\Sigma,\varphi): (S_0,\xi_0)\longrightarrow (S_1,\xi_1)$ we get natural 
	linear maps between the vector spaces 
	\begin{align}
	\Sigma^\varphi_\theta(\sigma_1,\sigma_0)&\longrightarrow \Sigma^\varphi_{\theta'}(\sigma_1,\sigma_0) \\
	[\mu] &\longmapsto [\langle \varphi^*\Lambda, \mu \rangle \cdot \mu] \ \ , 
	\end{align}
	where we added the subscripts $\theta$ and $\theta'$ to indicate the respective cocycles that enter the definition of the vector spaces $\Sigma^\varphi(?,?)$. These maps 
	induce maps between the corresponding coends and combine into a natural transformations 
	\begin{align}
	T_\Lambda(\Sigma,\varphi): T_{\theta'}(\Sigma,\varphi)\circ T_\Lambda(S_0,\xi_0)\Longrightarrow  T_\Lambda(S_1,\xi_1)\circ T_{\theta}(\Sigma,\varphi)\ \ . 
	\end{align}
	A straightforward computation shows that this defines a natural transformation of 2-functors. 
	Furthermore, it is clear how to equip $T_\Lambda$ with the structure of a symmetric monoidal transformation.
	Finally, we observe that $T_\Lambda$ is even a symmetric monoidal equivalence because $T_{-\Lambda}$ provides a weak inverse.       
\end{proof}
\begin{remark}
	In the same way an $n-2$-chain $\Omega$ satisfying $\operatorname{d}\Omega=\Lambda'-\Lambda$ induces symmetric monoidal 
	modifications between the natural transformations $T_\Lambda$ and $T_{\Lambda'}$. We do not spell 
	out the details here. 
\end{remark}

\section{Aspherical targets and transgression\label{sectransgression}}
We arrive at an interesting special case of Definition~\ref{defhqftext} if we choose as the target space of our sigma model the classifying space of a finite group.

\begin{definition}
	For a  finite group $G$ we define by $G\text{-}\Cob(n,n-1,n-2) := BG\text{-}\Cob(n,n-1,n-2)$ the \emph{symmetric monoidal category of $G$-bordisms} (there is a slight abuse of notation because $G\text{-}\Cob(n,n-1,n-2)$ could also describe bordisms with maps to the discrete space $G$; however that would not be interesting from the point of view of homotopy theory). We call an extended homotopy quantum field theory
	\begin{align}
	Z: G\text{-}\Cob(n,n-1,n-2) \to \Tvs
	\end{align} an \emph{extended $G$-equivariant topological quantum field theory}. 
	\end{definition}

\noindent In the non-extended case these appear as homotopy quantum field theories with aspherical targets in \cite{turaevhqft}. Three-dimensional extended $G$-equivariant topological field theories are discussed in \cite{maiernikolausschweigerteq} using the language of principal fiber bundles and with an emphasis on theories of Dijkgraaf-Witten type. A definition of extended $G$-equivariant topological field theories of arbitrary dimension and a detailed investigation of the three-dimensional case including a relation to equivariant modular categories is given in \cite{extofk}.

The goal of this section is to describe the extended equivariant topological quantum field theories associated to a cocycle on the classifying space of a finite group by transgression of this cocycle.

\subsection{Classification of 2-line bundles}
Let us briefly recall the definition of a line bundle:
A \emph{vector bundle over a groupoid $\Gamma$} is a functor $\rho : \Gamma \to \fvs$. Although in algebraic terms this is just a representation of $\Gamma$, the geometric viewpoint has proven to be profitable, see e.g. \cite{willterongerbesgrpds} or \cite{schweigertwoikeofk}. 

A line bundle $L :  \Gamma \to \fvs$ is a vector bundle for which all fibers, i.e. images $\rho(x)$ for $x\in \Gamma$, are one-dimensional vector spaces. Formulated differently, a line bundle takes values in the maximal Picard subgroupoid of $\fvs$. The Picard groupoid corresponding to $\fvs$ is equivalent to the category $\C \DS \C^\times$ with one object $\C$ and $\C^\times$ as endomorphisms. Hence, we
can factor every line bundle up to a natural isomorphism as
\begin{equation}\label{EQ: Classification of line bundles}
\begin{tikzcd}
\, \Gamma  \ar{r}{L}\ar{d}{\tilde{L}} & \fvs \\
\C\DS \C^\times \ar[hookrightarrow]{ru}{} &
\end{tikzcd}.
\end{equation} 
This shows that the groupoid of line bundles over $\Gamma$ is equivalent to the groupoid $[\Gamma,\C\DS\C^\times]$ of functors $\Gamma \to \C\DS\C^\times$. Hence, line bundles over $\Gamma$ are classified by 
\begin{align}
\pi_0[\Gamma,\C\DS\C^\times] =  [|B\Gamma|,K(\C^\times,1)] = H^1(\Gamma;\C^\times) \, 
\end{align}
where 
\begin{itemize}
	\item $|B\Gamma|$ is the geometric realisation of the nerve $B\Gamma$ of $\Gamma$,
	\item  $K(\C^\times,1)$ is the aspherical Eilenberg-MacLane space with fundamental group $\C^\times$
	\item and $H^1(\Gamma;\C^\times)$ the first groupoid cohomology with coefficients in $\C^\times$.
\end{itemize}
Categorifying the definition of a vector bundle over a groupoid, we arrive at the following notion, see \cite{swpar}:

\begin{definition}
A \emph{2-vector bundle over a groupoid $\Gamma$} is a 2-functor
\begin{align*}
\rho : \Gamma \to \Tvs \ ,
\end{align*}
where we consider $\Gamma$ as a 2-category with only trivial 2-morphisms. 
A \emph{2-line bundle over $\Gamma$} is a 2-vector bundle which takes values in the full Picard sub-2-groupoid of $\Tvs$.
\end{definition}
A 2-vector space is invertible with respect to the Degline tensor product if and only if it is 1-dimensional, i.e. equivalent to the category of vector spaces. Every linear functor between two 1-dimensional 2-vector spaces can be described up to natural isomorphism by a vector space. The functor is invertible if and only if this vector space is 1-dimensional.
This shows that the Picard 2-subgroupoid of $\Tvs$ is equivalent to $\fvs\DS \id \DS \C^\times$.  
Hence, for every 2-line bundle $L: \Gamma \to \Tvs$ there is a diagram 

\begin{equation}\label{EQ: Classification of 2-line bundles}
\begin{tikzcd}
\, \Gamma \ar{r}{L}  \ar{d}{\tilde{L}} & \fvs \\
\fvs\DS \id \DS \C^\times \ar[hookrightarrow]{ru}{}  &
\end{tikzcd}
\end{equation} 
commutative up to a 2-isomorphism.

Using the higher categorical analogue of \eqref{EQ: Classification of line bundles}, we arrive at a classification of 2-line bundles in terms of
\begin{align}
\pi_0[\Gamma,\fvs\DS \id \DS \C^\times] = [|B\Gamma|, K(\C^\times,2)] = H^2(\Gamma;\C^\times)\ .
\end{align} 
For a 2-functor $\tilde{L}:  \Gamma \to \fvs\DS \id \DS \C^\times$ there is only one choice for the definition on objects, 1-morphisms and 2-morphisms. The only non-trivial information contained in such a 2-functor are the coherence isomorphisms, which amount to a complex number $\alpha(g_1,g_2) \in \C^\times$ for every pair of composable morphisms $g_1$ and $g_2$ in $\Gamma$. These numbers form a groupoid 2-cocycle, see \cite{kirrilovg04} for the case of equivariant linear categories, \cite{swpar} for general 2-vector bundles and \cite{MS} for a detailed discussion of 2-line bundles.   
Given a 2-line bundle we can construct the corresponding 2-coycle by first choosing trivializations and then calculating the number corresponding to the coherence isomorphisms.
  
\subsection{A reminder on transgression}
Let us briefly recall the concept of transgression, see e.g. \cite{willterongerbesgrpds}: 
Let $M$ be an $\ell$-dimensional closed oriented manifold with fundamental class $\sigma$. For a topological space $T$ and a class in $H^k(T;\U(1))$ with $k\ge \ell$ represented by a cocycle $\theta$ we can define the class $\tau_M \theta \in H^{k-\ell} (T^M; \U(1))$ as being represented by the cocycle given by
\begin{align}
(\tau_M \theta) (\lambda) := (\text{ev}^* \theta) (\lambda \times \sigma) 
\end{align} for any $k-\ell$-simplex $\lambda : \Delta_{k-\ell} \to T^M$, where $\text{ev} : T^M \times M \to T$ is the evaluation map. Here $T^M$ is the space of maps $M\to T$ equipped with the compact-open topology. 
This gives rise to a map 
\begin{align}
\tau_M : H^k(T;\U(1)) \to H^{k-\ell} (T^M; \U(1)),
\end{align} the so-called \emph{transgression}. If $T$ is aspherical, then $T^M$ is equivalent to the groupoid $\Pi(M,T)$ of maps from $M \to T$ with equivalence classes of homotopies as morphisms. In that case the transgression can be seen to take values in $H^{k-\ell} (\Pi(M,T); \U(1))$, i.e. in groupoid cohomology. 

\subsection{Invariants of closed oriented manifolds equipped with bundles and transgression}
By Theorem~\ref{mainthm} we get for any flat $n-1$-gerbe on the classifying space of a finite group $G$, i.e. a cocycle $\theta \in Z^n(BG;\U(1))$ an extended $G$-equivariant topological quantum field theory $Z_\theta : G\text{-}\Cob(n,n-1,n-2) \to \TwoVect$. As any extended $G$-equivariant topological quantum field theory, it gives us an invariant for closed oriented manifolds of dimension $n$, $n-1$ and $n-2$ equipped with a map from the manifold to $BG$, i.e. a $G$-bundle over this manifold. 
Here invariant means that two manifolds with principal bundle $(M,\psi: M \longrightarrow BG)$ and
$(M',\psi': M' \longrightarrow BG)$ are assigned the same object (up to equivalence) if they are 
related by a diffeomorphism $f: M\longrightarrow M'$ combined with a homotopy $h\colon f^*\psi'\longrightarrow \psi$.

For a closed oriented $n$-dimensional manifold $M$, this invariant is a complex number given by the function
\begin{align}
Z_\theta(M,?) : \Pi(M,BG) \to \C, \quad \psi \mapsto \langle \psi^*\theta,\mu_M \rangle
\end{align} on the groupoid $\Pi(M,BG)$. This function is constant on isomorphism classes, i.e. it is a 0-cocycle in the cohomology of the groupoid $\Pi(M,BG)$. This cocycle is given by the transgression of $\theta$. More precisely, $Z(M,?) \in H^0 ( \Pi(M,BG); \U(1))$ is the image of $\theta$ under the transgression map $\tau : H^n(M;\U(1)) \to H^0 ( \Pi(M,BG); \U(1))$. 

We will now show that the invariant obtained from $Z_\theta$ for manifolds of dimension $n-1$ and $n-2$ equipped with bundles can also be described by an appropriate transgression of $\theta$.

Evaluating the primitive theory $Z_\theta$ on a closed $n-1$-dimensional $\Sigmait$ manifold gives a line bundle 
$Z_\theta (\Sigmait,?) : \Pi (\Sigmait , BG) \to \fvs$. To see this, we see $Z_\theta$ as a non-extended $G$-equivariant topological quantum field theory, see Remark \ref{Rem: Non extended theory}, and apply \cite[Proposition~2.10]{schweigertwoikeofk}. 

\begin{proposition}
Let $G$ be a finite group and $\theta \in Z^n(BG;\U(1))$. Then for any $n-1$-dimensional closed oriented manifold $\Sigmait$ the class $\langle Z_\theta(\Sigmait,?)\rangle \in H^1 (\Pi(\Sigmait,BG);\U(1))$ describing the line bundle $Z_\theta (\Sigmait,?) : \Pi (\Sigmait , BG) \to \fvs$ is given by
\begin{align} 
\langle Z_\theta(\Sigmait,?)\rangle = \tau_\Sigmait \theta,
\end{align} i.e. by the transgression $\tau_\Sigma \theta$ of $\theta$ to $\Pi(\Sigmait,BG)$. 
\end{proposition}

\begin{proof}
By abuse of notation we will denote the non-extended theory that $Z_\theta$ gives rise to in the sense of Remark~\ref{Rem: Non extended theory} also by $Z_\theta$. 
We fix a fundamental cycle $\sigma_\Sigmait$ of $\Sigmait$. This induces a linear isomorphism 
\begin{align}
Z_\theta(\Sigmait,\varphi)\to \C \\
[\sigma_\Sigmait] \ \ \mapsto 1
\end{align}
for all $\varphi : \Sigmait \to BG$. Consider a morphism $h: \varphi_1\to \varphi_2 $ in $\Pi(\Sigmait,BG)$, i.e. a map $h\colon \Sigmait \times [0,1]\rightarrow BG$. We can factor $h$ as 
\begin{equation}
\begin{tikzcd}
\, \Sigmait \times [0,1] \ar[swap]{d}{\id\times h} \ar{r}{h} & BG \\
\Sigmait \times BG^\Sigmait \ar[swap]{ru}{\text{ev}\circ T} &
\end{tikzcd} \ ,
\end{equation} 
where we denote by $T: \Sigma \times \Sigma^{BG} \rightarrow \Sigma^{BG} \times \Sigma$ the flip map and the image of $h$ under the adjunction $\Sigmait \times ? \dashv \ ?^\Sigmait $ again by $h$.
The cocycle ${Z}(\Sigmait, ?)$ evaluated on $h$ is given by \begin{align}
{Z_\theta}(\Sigmait, ?)(h)&= \langle h^* \theta , (-1)^{\text{dim} (\Sigmait)}\sigma_\Sigmait \times [0,1]  \rangle \\
 &= \langle \text{ev}^* \theta ,(-1)^{\text{dim}(\Sigmait) }T_*\sigma_\Sigmait \times h \rangle \\
 &= \langle \text{ev}^* \theta ,h \times \sigma_\Sigmait \rangle \\
 &= \tau_\Sigmait \theta \ ,  
\end{align}
where we used $T_*(a\times b)= (-1)^{\text{deg}(a)\cdot \text{deg}(b)}b\times a$, see e.g. \cite[Section~3.B]{AT}. 
\end{proof}

\noindent By \cite[Proposition~2.5]{extofk} we obtain by evaluation of $Z_\theta$ on an $n-2$-dimensional closed oriented manifold $S$ a representation $Z_\theta(S,?) : \Pi_2(S,BG) \to \TwoVect$ of the second fundamental groupoid of the mapping space $BG^S$ of maps from $S$ to $BG$. Since $BG$ is aspherical, this reduces to a 2-line bundle \begin{align} Z_\theta(S,?) : \Pi(S,BG)=\Pi_1(S,BG) \to \TwoVect\end{align} over the groupoid of maps $S \to BG$ with equivalence classes of homotopies as morphisms, i.e. the groupoid of $G$-bundles over $S$. This is accomplished by pulling back along a fixed equivalence
\begin{align}
\widehat{?} : \Pi_1(S,BG) \to \Pi_2(S,BG) \end{align} being the identity on objects and sending a class $h$ of homotopies to an arbitrary, but fixed representative $\widehat{h}$. The coherence isomorphisms for $\widehat{?}$ are unique.

\begin{theorem}\label{Theorem: Transgression}
	Let $G$ be a finite group and $\theta \in Z^n(BG;\U(1))$. Then for any $n-2$-dimensional closed oriented manifold $S$ the class $\langle Z_\theta(S,?)\rangle \in H^2 (\Pi(S,BG);\U(1))$ describing the 2-line bundle $Z_\theta (S,?) : \Pi (S , BG) \to \Tvs$ is given by
	\begin{align} 
	\langle Z_\theta(S,?)\rangle = \tau_S \theta, 
	\end{align} i.e. by the transgression $\tau_S \theta$ of $\theta$ to $\Pi(S,BG)$. 
	\end{theorem}

\begin{proof}
	We need to compute $\widetilde{Z_\theta(S,?)} : \Pi (S,BG) \to \FinVect // \id // \mathbb{C}^\times$. To this end, we define for each $\xi \in \Pi(S,BG)$ the equivalence
\begin{align} 
\chi_ \xi : Z_\theta(S,\xi) \to \FinVect, \quad V * \sigma \mapsto V 
\end{align} 
and fix the choice of a fundamental cycle $\sigma_S$ of $S$ to obtain a weak inverse
\begin{align} 
\chi_\xi^{-1} : \FinVect \to Z_\theta(S,\xi), \quad V \mapsto V* \sigma_S.
\end{align}
Next for any morphism $h: \xi_0 \to \xi_1$ in $\Pi(S,BG)$ we define the vector space $V_{\widehat{h}}$ by the (weak) commutativity of the square 
		\begin{equation}
	\begin{tikzcd}
Z_\theta(S,\xi_0) \ar{rr}{Z(S\times [0,1],\hat{h})} \ar[swap]{dd}{\chi_{\xi_0}} & & Z_\theta(S,\xi_1)  \\
	& & \\
	\FinVect \ar{rr}{?\otimes V_{\widehat{h}}} & & \FinVect \ar[swap]{uu}{\chi_{\xi_1}^{-1}}
	\end{tikzcd}, 
	\end{equation}
	i.e.
	\begin{align}
	V_{\widehat{h}} = \chi_{\xi_1} \int^{\sigma \in \Fund(S)} (S\times [0,1])^{\widehat h} (\sigma,\sigma_S) * \sigma \cong (S\times [0,1])^{\widehat h} (\sigma_S,\sigma_S).\label{eqnvsvh}
	\end{align} Note that we have a canonical isomorphism
	\begin{align}
	V_{\widehat h} \to \mathbb{C}, \quad (-1)^{\dim S} \sigma_S \times [0,1] \mapsto 1.\label{eqnisovhC}
	\end{align} For two composable morphisms $h$ and $h'$ in $\Pi(S,BG)$ we denote the composition by $h'h$ and obtain the 2-isomorphism
	\begin{align} Z_\theta(S\times [0,1],\widehat{h}) Z_\theta(S \times [0,1],\widehat{h'})  \xrightarrow{\substack{\text{coherence} \\  \text{of $Z_\theta$}}}  Z_\theta(S\times [0,1],\widehat{h'} \widehat{h}) 
	\xrightarrow{\substack{\text{evaluation of $Z_\theta$ on} \\  \widehat{h'} \widehat{h} \simeq \widehat{h'h}}} 
	Z_\theta(S\times [0,1],\widehat{h'h}).
	\end{align} By \eqref{eqnvsvh} this amounts to a map
	\begin{align}
	V_{\widehat h} \otimes V_{\widehat{h'}} \to V_{\widehat{h'}\widehat{h}} \to V_{\widehat{h'h}} ,\end{align} which by means of \eqref{eqnisovhC} can be seen as an automorphism of $\mathbb{C}$, i.e. an invertible complex number. By construction this is the number $\alpha_{\widetilde{Z_\theta(S,?)}} (h,h')$, i.e. the evaluation of 
	$\alpha_{\widetilde{Z_\theta(S,?)}} \in Z^2(\Pi(S,BG);\U(1))$ on the 2-simplex defined by the composable pair $(h,h')$. 
	
	By definition of $Z_\theta$ we find
	\begin{align}
	\alpha_{\widetilde{Z_\theta(S,?)}} (h,h') = \langle H^* \theta , \nu \rangle,
	\end{align} where
	\begin{itemize}
		\item $\nu$ is a fundamental cycle of $S\times [0,1]^2$ such that (see equation \eqref{EQ: Condition Fundamental cycle on manifolds with corners})\small
\begin{align}
\partial \nu = (-1)^{\dim S} \sigma_S \times \left( \{0\} \times [0,1]-\{1\} \times [0,1]-[0,1]\times \{0\}+[0,1/2]\times \{1\} +[1/2,1]\times \{1\}  \right)\label{eqndelnu}
\end{align}\normalsize 
		
		\item and $H: \widehat{h'} \widehat{h} \to \widehat{h'h}$ is a homotopy relative boundary. 
		
		\end{itemize}
Using the triangulation of $[0,1]^2$ given by	
	
\begin{center}
	\begin{overpic}[width=8cm,
		scale=0.3]{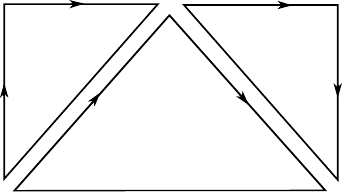}
\put(48,18){{$\nu_\Delta$}}
\put(10,40){{$\nu_-$}}
\put(85,40){{$\nu_+$}}
	\end{overpic}	 	
\end{center}
we get a fundamental cycle $\nu_\Box = \nu_- + v_\Delta + \nu_+$ of $[0,1]^2$.
Then $\nu :=  \sigma_S \times \nu_\Box$ satisfies \eqref{eqndelnu}. To get a representative for $H$ we pick a 2-simplex $\widetilde H : \Delta_2 \to BG^S$ such that $\partial_0\widetilde H = \widehat{h'}$, $\partial_1 \widetilde H = \widehat{h}$ and $\partial_2 \widetilde H = \widehat{h'h}$. This is possible since $BG^S$ is aspherical. Using the map $\Box : [0,1]^2 \to BG^S$ sketched in figure \ref{Fig:Sketch Box map} we define
\begin{align}
 H : S \times [0,1]^2 \xrightarrow{\id \times \Box} S \times BG^S \xrightarrow{\text{ev}\circ T} BG,
\end{align} where $\text{ev}: BG^S \times S \to BG$ denotes the evaluation. 
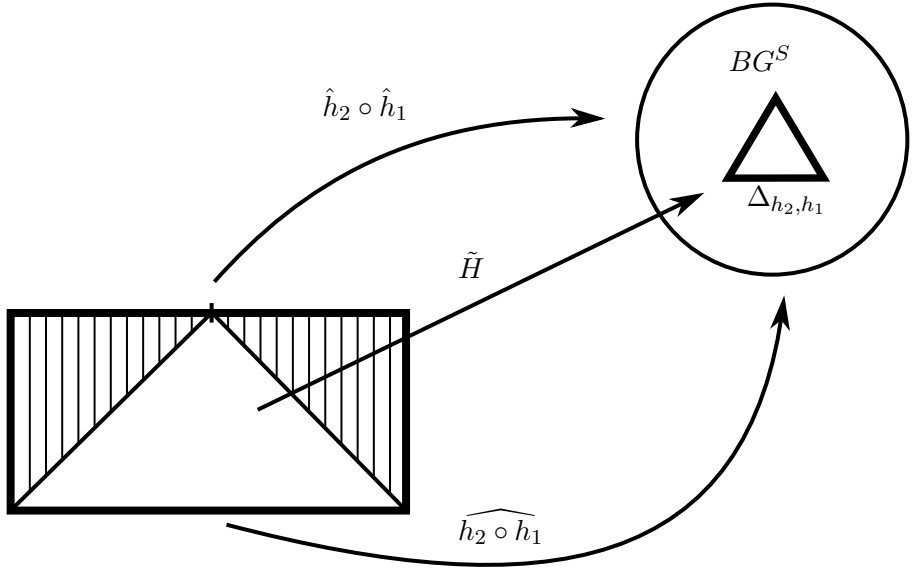
\begin{figure}[h]
	\centering
	\begin{overpic}[width=12cm,
		scale=0.3]{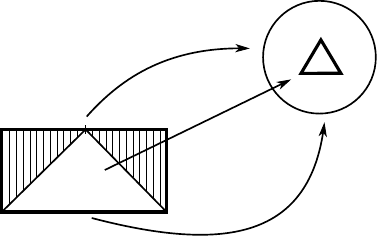}
		\put(35,50){{$\hat h_2 \circ \hat h_1$}}
		\put(50,3){{$\widehat{h_2 \circ  h_1}$}}
		\put(80,55){{$BG^S$}}
		\put(82,40){{$\Delta_{h_2, h_1}$}}
		\put(50,32){{$\tilde{H}$}}
	\end{overpic}
	\caption{Sketch for the definition of $\Box \colon [0,1^2]\rightarrow BG^S$. The map is constant along the vertical lines.}
	\label{Fig:Sketch Box map}	
\end{figure} 

Now we find 
	\begin{align}
	\alpha_{\widetilde{Z_\theta(S,?)}} (h,h') &= \langle H^* \theta , \nu \rangle \\ &= \langle \theta,H_*( \sigma_S \times \nu_-)\rangle +  \langle \theta,H_*( \sigma_S \times\nu_\Delta)\rangle +\langle \theta,H_*( \sigma_S \times \nu_+)\rangle.
	\end{align} Without loss of generality we can work with a normalized representative for $\theta$ which then vanishes on the degenerate simplices $H_*\nu_-$ and $H_*\nu_+$. Hence, we are left with
	\begin{align}
		\alpha_{\widetilde{Z_\theta(S,?)}} (h,h') = \langle \theta,H_*\nu_\Delta\rangle = \langle \text{ev}^* \theta,T_*(\sigma_S \times \Box_* \nu_\Delta)\rangle = \langle \text{ev}^* \theta, \widetilde H \times \sigma_S \rangle = \tau_S \theta (h,h').
	\end{align} This proves the assertion.
	\end{proof}

\section{Twisted equivariant Dijkgraaf-Witten theories\label{twisteddwtheories}}
Dijkgraaf-Witten theory is an extended topological quantum field theory that can be associated to a finite group $G$. It is based on the ideas in \cite{dijkgraafwitten} and was developed further in \cite{freedquinn} and \cite{morton1}. In this section we explain how twisted Dijkgraaf-Witten theories are obtained by orbifoldization of extended equivariant topological quantum field theories coming from cocycles.

For $\theta \in H^n(BG; \U(1))$ the associated extended field theory $Z_\theta : G\text{-}\Cob(n,n-1,n-2)\to \Tvs$ can be seen as a classical gauge theory with topological twist $\theta$. The corresponding quantized theory can be obtained by the orbifold construction given in \cite{extofk}, i.e. by means of applying the orbifoldization functor
$?/G$ from $n$-dimensional extended $G$-equivariant topological field theories to $n$-dimensional extended ordinary (i.e. non-equivariant) topological field theories. The orbifoldization combines a sum over twisted sectors with the computation of (homotopy) invariants by means of the parallel section functor developed in \cite{swpar}.

\begin{definition}\label{DW-theories}
For a finite group $G$ and $\theta\in H^n(BG; \U(1))$ the \emph{$n$-dimensional $\theta$-twisted Dijkgraaf-Witten theory 
\[
\DW _\theta : \Cob(n,n-1,n-2) \to \TwoVect
\]
with gauge group $G$} is defined to be the orbifold theory $Z_\theta / G$ of the extended field theory $Z_\theta : G\text{-}\Cob(n,n-1,n-2)\to \Tvs$ associated to $\theta$.  
\end{definition}

\noindent This generalizes the description of the non-extended twisted Dijkgraaf-Witten theory in \cite[Example~3.48]{schweigertwoikeofk} and the construction of three-dimensional extended twisted Dijkraaf-Witten theory in \cite{morton1}, as we will explain in more detail below. 

Let us first give the following formula for the number of simple objects in the category $\DW_\theta(\mathbb{T}^{n-2})$ obtained by evaluation of the twisted Dijkgraaf-Witten theory on the $n-2$-dimensional torus $\mathbb{T}^{n-2}$. It follows as a special case from \cite[Theorem~4.21]{extofk}: 

\begin{proposition}\label{satznumberofsimples}
	For a finite group $G$, $\theta\in H^n(BG; \U(1))$
	\begin{align}
	\# \{ \text{simple objects of $\DW_\theta(\mathbb{T}^{n-2})$} \} = \frac{1}{|G|} \sum_{\substack{  g_1,\dots,g_n \in G \\ \text{mutually commuting}  }} \langle  \psi_{g_1,\dots,g_n}^* \theta, \mu_{\mathbb{T}^n}   \rangle    \ , \label{eqnnumberofsimples}
	\end{align} where
	\begin{itemize}
		
		
		\item $\psi_{g_1,\dots,g_n}: \mathbb{T}^n \to BG$ is a classifying map for the $G$-bundle $P$ over $\mathbb{T}^n$ specified by the holonomy values $g_1,\dots,g_n \in G$,
		
		\item $\mu_{\mathbb{T}^n}$ is the fundamental class of the torus.
	
		\end{itemize}
	\end{proposition}


\noindent It is worth looking at Definition~\ref{DW-theories} for $n=2$:



\begin{proposition}\label{dwtheorydim2}
For any finite group $G$ and $\theta\in H^2(BG; \U(1))$ the evaluation of the topological quantum field theory $\DW _\theta = Z_\theta / G: \Cob(2,1,0) \to \TwoVect$ on the point is given by the category of $\theta$-twisted projective representations of $G$.
\end{proposition}
 
\begin{proof}
The groupoid $\Pi(\star, BG)$ is equivalent to the groupoid $\star \DS G$ with one object and automorphism group $G$. By Theorem \ref{Theorem: Transgression} (note that the transgression is the identity in that case) we find that the 2-vector bundle $Z_\theta(\star,?)$ is given by
\[
 \star \DS G \overset{\theta}{\to} \fvs \DS \id \DS \C^\times \to \Tvs \ ,
\] where $\theta$ is understood as a 2-functor. 
According to the definition of the orbifold construction in \cite{extofk},
 the 2-vector space $\DW _\theta(\star)$ is given by the category of 1-morphisms from the trivial line bundle over $ \star \DS G$ to $Z_\theta(\star,?)$, i.e. by the parallel sections of $Z_\theta(\star,?)$. 
Spelling this out we see that $\DW _\theta(\star)$ is the category of projective representation twisted by $\theta$, see \cite[Section 3.4]{MS} for more details. 
\end{proof}

\begin{example}
	Given the explicit description of $\DW _\theta(\star)$ provided by Proposition~\ref{dwtheorydim2} in the two-dimensional case, we can compute the number of irreducible $\theta$-twisted representation of $G$ by using Proposition~\ref{satznumberofsimples}. The right hand side of \eqref{eqnnumberofsimples}, i.e. the value of $\theta$-twisted Dijkgraaf-Witten theory on the torus, already appears in (6.40) of \cite{dijkgraafwitten}, although we should note that the reasoning in the proof of Proposition~\ref{satznumberofsimples} is only valid because we have described twisted two-dimensional Dijkgraaf-Witten theory as an \emph{extended} quantum field theory. 
Now  \eqref{eqnnumberofsimples} reduces to
	\begin{align}
	\# \{ \text{irreducible $\theta$-twisted representation of $G$} \}  = \frac{1}{|G|}\sum_{gh=hg} \frac{\theta(h,g)}{\theta(g,h)}
	\end{align} and hence to the result found in \cite[Corollary 13]{willterongerbesgrpds} by algebraic methods. 
	\end{example}

\noindent As a special case of Theorem~\ref{Theorem: Equivariant category on S1} below we will find that the evaluation $\DW_\theta(\mathbb{S}^1)$ of the 3-2-1-dimensional $\theta$-twisted Dijkgraaf-Witten theory on the circle is given by the category of $\tau_{\mathbb{S}^1}\theta$-twisted representations of the action groupoid $G\DS G$. By \cite[Proposition~8 and Theorem~17]{willterongerbesgrpds} this category is the representation category of the twisted Drinfeld double \cite{definition twisted Drinfeld double}. Hence, we have proven:
\begin{theorem}[]
For any finite group $G$ and $\theta\in H^3(BG; \U(1))$ the evaluation of $\DW _\theta = Z_\theta / G$ on the circle is given by the representation category of the $\theta$-twisted Drinfeld double of $G$. 
\end{theorem}

\noindent This result implies for instance that we can get from Proposition~\ref{satznumberofsimples} an easy topological proof of the formula for the number of irreducible representations of the $\theta$-twisted Drinfeld double given in \cite[Theorem 21]{willterongerbesgrpds}. 

Moreover, we see that the $\theta$-twisted Dijkgraaf-Witten theory of Definition~\ref{DW-theories} generalizes the $\theta$-twisted Dijkgraaf-Witten theory given in \cite{morton1} for the 3-2-1-dimensional case to arbitrary dimension because they yield the same modular category upon evaluation on the circle, which is sufficient by the classification result of \cite{BDSPV15}.

As a generalization of the orbifold construction, we get for any morphism $\lambda : H \to J$ of finite groups a \emph{pushforward map} $\lambda_*$ from $H$-equivariant to $J$-equivariant topological field theories, see \cite[Section~6]{schweigertwoikeofk} for the non-extended case and \cite[Section~3.3]{extofk} for the extended case needed here. We use this construction to define a new class of extended field theories generalizing work of \cite{maiernikolausschweigerteq}:

\begin{definition}\label{defgendwtheory}
Let $\lambda : H \to J$ be a morphism of finite groups and $\theta\in H^n(BH;\U(1))$.
The \emph{$\theta$-twisted $J$-equivariant Dijkgraaf-Witten theory} $\DW^\lambda_\theta := \lambda_* Z_\theta : J\text{-}\Cob(n,n-1,n-2) \to \TwoVect$ is defined to be the pushforward of $Z_\theta$ along $\lambda$. 
\end{definition}

Consider a 3-dimensional extended $J$-equivariant topological quantum field theory $Z$. We pick a base point of $\mathbb{S}^1$ and identify a principal bundle with its holonomy around $\mathbb{S}^1$ in the direction induced by the orientation. This identifies the groupoid of $J$-bundles on $\mathbb{S}^1$ with the action groupoid $J\DS J $.
It is convenient to form the category
\begin{align}
\mathcal{C}^Z = \bigoplus_{j\in J} Z(\mathbb{S}^1,j),
\end{align}   
where we identify a bundle with its holonomy. 
As proven in \cite[Theorem~4.33~(b)]{extofk}, $\mathcal{C}^Z$ carries the structure of a $J$-multimodular category. By \cite[Theorem~4.33~(a)]{extofk}, $\mathcal{C}^Z$ is $J$-modular if and only if the tensor unit of $\mathcal{C}^Z$ is simple. 
Before calculating $\mathcal{C}^{\text{DW}_\theta^\lambda}$ we recall the following definition:

\begin{definition}
Let $\Gamma$ be a groupoid and $\alpha \in H^2(\Gamma; \C^\times)$ a cohomology class represented by a normalized 2-cocycle (also denoted by $\alpha$). A \emph{projective functor} $F     :     \Gamma \to \fvs$ with respect to $\alpha$ assigns to every object $g\in \Gamma$ a vector space $F(g)\in \fvs$ and to a morphism $\varphi     :     g_1 \to g_2$ a linear map $F(\varphi)    :     F(g_1)\to F(g_2)$, such that
\begin{align*}
F(\varphi_2) \circ F(\varphi_1)= \alpha(\varphi_2,\varphi_1) F(\varphi_2 \circ \varphi_1)
\end{align*}
for composable morphisms $\varphi_1,\varphi_2 \in \Gamma$ and
\begin{align}
F(\id_g)= \id_{F(g)} \
\end{align}
for all $g\in \Gamma$.
Natural transformation between projective functors can be defined as usual.  
We denote by $[\Gamma, \fvs]^\alpha$ the 2-vector space of projective functors with respect to $\alpha$.   

\end{definition}

\noindent By chasing through the definition of parallel sections of a 2-vector bundle in \cite{swpar} and using (3.15) and (3.16) in \cite{MS}, we can observe the following: 

\begin{lemma}\label{Lemma: Charkterisation of twisted functors}
Given a normalized groupoid 2-cocycle $\alpha$ described by a 2-functor \begin{align} \alpha     :     \Gamma \to \fvs \DS \id \DS \C^\times \to \Tvs, \end{align} the 2-vector space $[\Gamma, \fvs]^\alpha$ of $\alpha$-projective functors can be equivalently described as
the 2-vector space of parallel section of $\alpha$ seen as a 2-vector bundle in the sense of \cite{swpar}, i.e. as 2-vector space of 2-natural transformations from the trivial 2-vector bundle to $\alpha$.   
\end{lemma}

\begin{theorem}\label{Theorem: Equivariant category on S1}
Let $\lambda     :     H \to J$ be a morphism of finite groups and $\theta\in H^3(BH;\U(1))$. 
\begin{enumerate}
	\item The category assigned to $\mathbb{S}^1$ by the extended $J$-equivariant $\theta$-twisted Dijkgraaf-Witten theory $\DW^\lambda_\theta$  is given by
\begin{align}\label{Equation: Equivariant category on S1}
\mathcal{C}^{\DW_\theta^\lambda} = \bigoplus_{j\in J} [\lambda_*^{-1}[j], \fvs]^{q_j^* \tau_{\mathbb{S}^1}\theta }\ , 
\end{align}
where $\lambda_*^{-1}[j]$ is the homotopy fiber over $j\in J$ under the functor $\lambda_*     :     H\DS H \to J\DS J$ induced by $\lambda$ and $q_j    :     \lambda_*^{-1}[j] \to H\DS H $ is part of the structure of the homotopy fiber. 
The category \eqref{Equation: Equivariant category on S1} carries a $J$-multimodular structure.

\item If $\lambda$ is surjective, \eqref{Equation: Equivariant category on S1} reduces to
\begin{align}\label{Equation: lambda surjective}
\mathcal{C}^{\DW_\theta^\lambda} = \bigoplus_{j\in J}\  [\lambda^{-1}(j)\DS \ker \lambda, \fvs]^{\tau_{\mathbb{S}^1}\theta|_{\ker \lambda}} \ ,
\end{align} where $\lambda^{-1}(j)$ is the preimage of $j\in J$ under the group morphism $\lambda$. The category \eqref{Equation: lambda surjective} is $J$-modular.
\end{enumerate}
\end{theorem}

\begin{proof}
	\begin{enumerate}
		\item 
Equation \eqref{Equation: Equivariant category on S1} can be obtained directly from Lemma~\ref{Lemma: Charkterisation of twisted functors} and Theorem~\ref{Theorem: Transgression}. The $J$-multi\-modu\-larity of $\mathcal{C}^{\DW_\theta^\lambda}$ follows from \cite[Theorem~4.33~(b)]{extofk} because the category comes from a $J$-equivariant topological field theory. 

\item If $\lambda    :     H\to J$ is surjective we can pick a set theoretical section of $s    :     J \to H$ of $\lambda$. Every object $(h, \tilde{j}     :     \lambda(h)\to j) \in \lambda_*^{-1}[j]$ in the homotopy fiber is isomorphic to $(s(\tilde{j})h s(\tilde{j})^{-1}, 1)$ by the isomorphism $s(\tilde{j})$. For this reason $\lambda_*^{-1}[j]$ is equivalent to its full subgroupoid consisting of objects of the form $(h,1)$ with $h\in \lambda^{-1}(j)$. A morphism $\tilde{h}    :     (h,1) \to (h'=\tilde{h} h \tilde{h}^{-1},1)$ satisfies
\begin{equation}
\begin{tikzcd}
\lambda (h) \ar[swap]{rd}{1} \ar{rr}{\lambda (\tilde{h})} & & \lambda (h') \ar{ld}{1} \\
 & j &
\end{tikzcd} \ . 
\end{equation}
This implies that $\tilde{h} \in \text{ker}(\lambda)$ and hence $\lambda_*^{-1}[j]\cong \lambda^{-1}(j)\DS \text{ker} (\lambda) $, which proves \eqref{Equation: lambda surjective}. In order to see that \eqref{Equation: lambda surjective} is in fact $J$-modular, it suffices to show that the unit of \eqref{Equation: lambda surjective} is simple. This follows from \cite[Proposition~4.29]{extofk}. Again, surjectivity of $\lambda$ enters.
\end{enumerate}
\end{proof}

\begin{remark}
\begin{enumerate}

\item 
For $\theta=0$ and surjective $\lambda$  this result agrees with \cite[Proposition~3.22]{maiernikolausschweigerteq}. Hence, our construction is a natural extension to non-trivial cocycles. 

\item Let us give the following physical interpretation of the techniques developed in this section: Suppose that we are given a finite group $G$ and a 3-cocycle $\omega \in H^3(BG;\U(1))$. This data gives rise to a 3-2-1-dimensional topological quantum field theory, namely the $\omega$-twisted Dijkgraaf-Witten theory $\DW_\omega$, see Definition~\ref{DW-theories}. Suppose now we are given a weak action of a finite group $J$ on $G$, see \cite[Definition~3.1]{maiernikolausschweigerteq}, preserving $\omega$. Such a weak action should be seen as a symmetry of our theory. We can ask whether it is possible to gauge this $J$-symmetry, i.e. whether we can find a $J$-equivariant topological quantum field theory whose neutral sector is given by $\DW_\theta$. 
To this end, we note as in \cite{maiernikolausschweigerteq} that the weak $J$-action on $G$ amounts to a short exact sequence
\begin{align}
1 \to G \to H \stackrel{\lambda}{\to} J \to 1\end{align} of finite groups. If there is a $\theta \in H^3(BH;\U(1))$ such that 
\begin{align}
\theta|_G=\omega,\label{eqnfindtheta}
\end{align} then indeed, by Theorem~\ref{Theorem: Equivariant category on S1} $\DW_\omega$ occurs as the neutral sector of $\DW_\theta^\lambda$. 
In general, there are obstructions for finding such a $\theta$ as in \eqref{eqnfindtheta}. Physically, this is the appearance of so-called \emph{'t Hooft anomalies}, see \cite{MSb} for more details.

\item
Our construction can be generalized as follows: Given a sequence of finite groups 
\begin{align}
G_0 \overset{\lambda_1}{\longrightarrow} G_1 \overset{\lambda_2}{\longrightarrow} \dots \xrightarrow{\lambda_{n}} G_n
\end{align}
and 3-cocycles $\theta_j \in H^3(BG_j;\U(1))$ for $0\le j\le n$ we can construct the $G_n$-equivariant topological quantum field theory
\begin{align}
{\lambda_{n}}_*(\cdots ({\lambda_2}_*({\lambda_1}_* Z_{\theta_0}\otimes Z_{\theta_1})\otimes \cdots )\otimes Z_{\theta_{n-1}}) \otimes Z_{\theta_n} \ .
\end{align}
Corresponding to this theory there exists a potentially interesting $G_n$-multimodular tensor category.  
\end{enumerate}
\end{remark}

\end{document}